\documentclass[11pt]{amsart}

\usepackage{amsmath, amsfonts, amssymb,amsthm}
\usepackage{amstext}
\usepackage{mathrsfs}

\usepackage{float,epsf,subfigure}
\usepackage[all,cmtip]{xy}
\usepackage{hyperref}
\usepackage{algorithm}
\usepackage{algorithmic}
\usepackage{mathtools}
\usepackage{float}
\usepackage{bm}
\theoremstyle{plain}
\newtheorem{theorem}{Theorem}[section]

\newtheorem{cor}[theorem]{Corollary}

\newtheorem{def-thm}[theorem]{Definition-Theorem}
\newtheorem{lemma}[theorem]{Lemma}

\newtheorem{defi}[theorem]{Definition}

\theoremstyle{definition}

\newtheorem{remark}[theorem]{Remark}

\def\min{\mathop{\mathrm{min}}}

\begin{document}
\title[Equidistribution Theory of Algebroid Mappings]{Algebroid Mappings and Their  Equidistribution Theory}
\author[X.-J. Dong]
{Xianjing Dong}

\address{School of Mathematical Sciences \\ Qufu Normal University \\ Qufu, Jining, Shandong, 273165, P. R. China}
\email{xjdong05@126.com}



\date{}
\maketitle \thispagestyle{empty} \setcounter{page}{1}

\begin{abstract}  
In this paper, the concept  of algebroid mappings of complex manifolds is introduced 
based on that a large number of   complex systems of PDEs admit   multi-valued  solutions that  
can be defined by a system of independent algebraic equations over the field of meromorphic functions. 
It is well-known that Nevanlinna theory is  an important     tool  in 
complex ODE theory. To develop a similar     tool applied to the study of  algebroid solutions of complex systems of PDEs, 
 one  explores  the  equidistribution  theory of algebroid mappings. 
 Via   uniformizating an  algebroid mapping,   we 
  obtain   a second main theorem
     of algebroid mappings from a complete K\"ahler manifold 
           into a complex  projective manifold provided that   some  certain conditions are imposed. 
  \end{abstract}

\tableofcontents
\medskip

\newpage
\setlength\arraycolsep{2pt}
\medskip

\section{Introduction}

\subsection{Motivation}~

Algebroid solutions of complex ODEs have been studied extensively  since the notion of  algebroid functions  was  introduced by H. Poincar\'e. 
 It has been discovered that  many  complex  ODEs have   algebroid solutions. 
For instance,   $4zW'^2-1=0$   is of   an algebroid solution $W=\sqrt{z}.$
What about  an algebroid solution of a complex system of   PDEs?
   Let $F_i(z, W_1,\cdots, W_d)$   $(i=1,\cdots,k)$ be   functions  
   of   $W_1,\cdots,W_d$   and their partial derivatives on $\mathbb C^m$ 
      over the field $\mathscr M(\mathbb C^m)$ of all 
   meromorphic functions on $\mathbb C^m,$ 
   where  $z=(z_1,\cdots,z_m)\in\mathbb C^m.$
    A solution  $(W_1, \cdots, W_n)$ of a  given  complex system of PDEs: 
	\begin{equation}\label{d0}
	F_i\big(z, W_1,\cdots,W_d\big)=0, \ \ \ \ i=1,\cdots,k
	\end{equation}
is called an algebroid solution, if it is defined by a system of $d$  independent   algebraic  equations 
\begin{equation}\label{d1}
P_i\big(z, W_1,\cdots, W_d\big)=0, \ \ \ \   i=1,\cdots, d
\end{equation}
in $W_1,\cdots, W_d$  over $\mathscr M(\mathbb C^m).$  
 Here, the ``independent" means that  
$$\left|\frac{\partial(P_1,\cdots, P_d)}{\partial(W_1,\cdots,W_d)}\right|\not\equiv0.$$
Note that  a  meromorphic solution of Eq. (\ref{d0}) is always    an algebroid solution. 
To  our knowledge,   a large number of  complex systems of  PDEs, particularly,  lots of complex systems of  algebraic PDEs,     have  
non-meromorphic   algebroid solutions.  

\noindent\textbf{Examples}
\begin{itemize} 
\item The complex system of PDEs 
			   $$\begin{cases}
	\frac{\partial W}{\partial z_1}=\frac{z_2}{2W} \\
	  \frac{\partial W}{\partial z_2}=\frac{z_1}{2W} \\
	    \frac{\partial^2 W}{\partial z_1\partial z_2}=\frac{1}{4W}
	   \end{cases}$$
on $\mathbb C^2$ has an algebroid solution  $$W=\sqrt{z_1z_2}.$$ 
\item The complex system of PDEs
   $$\begin{cases}
W_1\frac{\partial^2 W_1}{\partial z_1\partial z_2}+W_2\frac{\partial^2 W_2}{\partial z_1\partial z_2}+\frac{\partial W_1}{\partial z_1}\frac{\partial W_1}{\partial z_2}
+\frac{\partial W_2}{\partial z_1}\frac{\partial W_2}{\partial z_2}-\frac{1}{2}=0 \\
W_2\frac{\partial^2 W_1}{\partial z_1\partial z_2}+W_1\frac{\partial^2 W_2}{\partial z_1\partial z_2}+\frac{\partial W_1}{\partial z_1}\frac{\partial W_2}{\partial z_2}
+\frac{\partial W_2}{\partial z_1}\frac{\partial W_1}{\partial z_2}-1=0
   \end{cases}$$
on $\mathbb C^2$ has an algebroid solution  $(W_1, W_2)$ defined by  
$$W_1^2+W_2^2=z_1z_2, \ \ \ \    W_1W_2=z_1+z_2.$$ 
\item The complex system of PDEs 
			   $$\begin{cases}
	\frac{\partial^2 W_1}{\partial z_1\partial z_2}=\frac{1}{W_2} \\
	 \frac{\partial^2 W_2}{\partial z_1\partial z_2}=\frac{1}{W_1}
	   \end{cases}$$
on $\mathbb C^2$ has  algebroid solutions $(W_1, W_2)$ of the form 
  $$W_1=W_2=\pm2\sqrt{(\lambda z_1+c_1)(\lambda^{-1}z_2+c_2)},$$
  where $\lambda\not=0, c_1,c_2$ are arbitrary constants.
  \end{itemize} 

In the viewpoint of mappings, 
 any algebroid solution $(W_1,\cdots, W_d)$ of Eq. (\ref{d0})    corresponds  to  a multi-valued analytic  mapping
		$$F=\big[1: W_1:\cdots: W_d\big]: \ \mathbb C^m\to\mathbb P^d(\mathbb C).$$ 
	On the other hand, if the above $F$ is a multi-valued  mapping defined by Eq. (\ref{d1}), then $(W_1,\cdots, W_d)$ can be viewed  as an algebroid solution of a 
		complex system of  PDEs by differentiating  Eq. (\ref{d1}).  It is a   one-to-one correspondence.
This observation   motivates us to introduce the notion of algebroid mappings.

Let $M, N$ be  complex manifolds of complex dimensions $m, n,$ respectively. 
Let   $\mathscr M(M)$ denote the field of all  meromorphic functions on $M.$
We give the following definition of algebroid mappings (see Definition \ref{defi11} in Section \ref{s11}). 
	
\noindent\emph{Definition.}
 We say that a  multi-valued mapping  $F: M\to N$ is     an algebroid mapping,
    if  for  every point  $y\in f(x)$ with  all $x\in M,$ 
    $\xi_1\circ F,  \cdots, \xi_n\circ F$
 are   local solutions  of  a  system of $n$ independent  algebraic equations 
 $$P_i\big(x, W_1,\cdots, W_n\big)=0, \ \ \ \   i=1,\cdots, n$$
 in  $W_1,\cdots,W_n$  over $\mathscr M(M),$   where $\xi=(\xi_1,\cdots,\xi_n)$ is a  local holomorphic coordinate near $y.$

According to the definition,  we can see that algebroid mappings generalize the notions of  meromorphic mappings and algebroid functions.

\noindent\textbf{Examples}

	\begin{itemize} 
		\item  A  2-valued  rational algebroid mapping  
					 $$F=\left[\frac{\zeta_1}{\zeta_0}: \frac{\zeta_0\zeta_2}{\zeta_1^2}: \frac{\sqrt{\zeta_2}}{\sqrt{\zeta_1}}\right]: \ \mathbb P^2(\mathbb C)\to \mathbb P^2(\mathbb C).$$
						\item A  6-valued  transcendental algebroid curve 
					 $$F=\left[z: \sqrt{z}: \sqrt[3]{z}:e^z\right]: \ \mathbb C\to \mathbb P^3(\mathbb C).$$
			\item A  12-valued transcendental  algebroid mapping 
							 $$F=\big[z_1:  \sin z_2: \sqrt{z_1}\sqrt[4]{z_2} : \sqrt[3]{z_3}e^{z_4}\big]: \ \mathbb C^4\to \mathbb P^3(\mathbb C).$$
		\end{itemize} 	

It is well-known that Nevanlinna theory \cite{ru} plays a useful  role in the study of meromorphic and algebroid solutions of  complex ODEs (see \cite{H-X, Laine,  S-G}). 
It seems helpful to develop a value distribution theory of  algebroid mappings,  
as a theoretical tool of  investigating   algebroid solutions of a complex system of PDEs. 
Since R. Nevanlinna established  Nevanlinna theory,  many  scholars  such as 
G. R\'emoundos, H. Selberg, 
  E. Ullrich and G. Valiron, etc.  have been exploring    the value distribution theory of algebroid functions. 
  For   instance, G. R\'emoundos \cite{GR} proved  a Picard's theorem   which      
   says   that   
   a nonconstant $\nu$-valued algebroid function on $\mathbb C$   omits $2\nu$ distinct  values at most; and then,  
  G. Valiron  \cite{GV}  showed    without details 
      a five-value type theorem 
             which says   that  two nonconstant $\nu$-valued algebroid functions
   on $\mathbb C$
   are identical if they share $4\nu$+1 distinct  values ignoring multiplicities. 
     For  more details, we refer  the reader    to     J. Noguchi \cite{No0}, 
 Niino-Ozawa \cite{NO, Oz}, 
   J. Suzuki \cite{JS},  L. Selberg \cite{LS1, LS2}, N. Toda \cite{To1, To2} and  E. Ullrich \cite{EU};  and  see   
 Hu-Yang \cite{HY}, 
 He-Xiao \cite{H-X}, M. Ru \cite{ru0} and 
 Sun-Gao \cite{S-G}, etc. 

However, 
there is very  little   literature about   the value distribution theory of multi-valued analytic mappings, although the researchers  have a relatively profound  understanding of  single-valued analytic mappings (see \cite{No, ru, Sha}).
  In this  paper, we  shall explore  equidistribution theory  of algebroid mappings from a complete  K\"ahler manifold into a complex projective manifold.  
This is a generalization of Carlson-Griffiths equidistribution theory of holomorphic mappings from $\mathbb C^m$ into a complex projective manifold (see \cite{gri}). 

\subsection{Main Results}~

Let $M$ be a complex manifold, and $N$ a complex projective manifold. 

   \begin{theorem}[=Theorem \ref{cover}]
   A $\nu$-valued algebroid mapping $F: M\to N$ defines   a $\nu$-leaf complex manifold $\mathcal M_F,$ which  is a  $\nu$-sheeted ramified analytic covering of $M.$
      \end{theorem}

 \begin{cor}[=Corollary \ref{uniform}]\label{}
 An algebroid mapping $F: M\to N$ can be lifted to a meromorphic mapping $f:  \mathcal M_F\to N$ via the natural projection $\pi: \mathcal M_F\to M$  
such that $F=f\circ\pi^{-1}.$
\end{cor}

Let $(M, \alpha)$ be an $m$-dimensional  complete non-compact K\"ahler manifold with Ricci form $\mathscr R,$ and $(N, \omega)$ an $n$-dimensional   K\"ahler projective manifold.  Let $F: M\to N$ be a $\nu$-valued algebroid mapping.  
Referring  to  Sections \ref{61} and \ref{63}, one can define  Nevanlinna's functions $T_F(r, \omega), m_F(r, D), \overline{N}_F(r, D)$ and $T(r, \mathscr R)$
under different curvature conditions.  Since $N$ is projective, one can embed $N$ holomorphically into a complex projective space with Fubini-Study form $\omega_{FS}.$ 
Let  $\omega_{FS}|_N$ be the restriction of  $\omega_{FS}$ to $N.$
 Set
  $$\left[\frac{\omega_{FS}|_N}{\omega}\right]=\inf\big\{s>0: \  \omega_{FS}|_N\leq s\omega\big\}.$$

$a)$  $M$ has non-positive sectional curvature. Set  
\begin{equation}\label{ricci}
 \kappa(r)=\frac{1}{2m-1}\inf_{x\in B(r)}R(x),
 \end{equation}
   where   
    $$R(x)=\inf_{X\in T_xM, \ \|X\|=1}{\rm{Ric}}(X, X)$$
    is   the pointwise lower bound of Ricci curvature of $M,$ and $B(r)$ is the geodesic ball in $M$ centered at some  reference  point $o$ with radius $r.$

  \begin{theorem}[=Theorem \ref{main3}]  Let $M$ be a  
     complete non-compact  K\"ahler manifold   with non-positive sectional curvature.  
Let $(N,\omega)$ be a   K\"ahler  projective manifold  of  dimension not greater than  that  of $M.$ 
  Let $D_1,\cdots,D_q$ be effective divisors in general position on $N$ such that each $D_j$ is cohomologous to  $\omega.$
  Let  $F: M\rightarrow N$ be a  differentiably non-degenerate $\nu$-valued algebroid mapping.  
  Assume that $q\omega-{\rm Ric}(\omega^n)>0.$
  Then  for any  $\delta>0,$ there exists a subset $E_\delta\subset(1, \infty)$ of finite Lebesgue measure such that 
        \begin{eqnarray*}
&&\left(q-(2\nu-2)\left[\frac{\omega_{FS}|_N}{\omega}\right]\right)T_F(r,\omega)+T_F(r, K_N)+T(r, \mathscr R) \\
&\leq& \sum_{j=1}^q\overline N_F(r,D_j)+O\big(\log T_{F}(r, \omega)+\sqrt{-\kappa(r)}r+\delta\log r\big)
         \end{eqnarray*}
    holds for all  $r>1$ outside $E_\delta,$  where $\kappa$ is defined by $(\ref{ricci}).$          
\end{theorem}  

  \begin{cor}[=Corollary \ref{hco1}]\label{}    Let $M$ be a  
     complete non-compact  K\"ahler manifold   with non-positive sectional curvature. 
  Let $H_1,\cdots, H_q$ be hyperplanes  of  $\mathbb P^n(\mathbb C)$ in general position.  
  Let  $F: M\rightarrow \mathbb P^n(\mathbb C)$    be a  differentiably non-degenerate $\nu$-valued algebroid mapping.  
 Assume that $m\geq n.$ Then  for any  $\delta>0,$ there exists a subset $E_\delta\subset(1, \infty)$ of finite Lebesgue measure such that 
        \begin{eqnarray*}
&& \left(q-2\nu-n+1\right)T_F(r,\omega_{FS}) +T(r, \mathscr R) \\
&\leq& \sum_{j=1}^q\overline N_F(r,H_j)+O\big(\log T_{F}(r, \omega_{FS})+\sqrt{-\kappa(r)}r+\delta\log r\big)
        \end{eqnarray*}
    holds for all  $r>1$ outside $E_\delta,$  where $\kappa$ is defined by $(\ref{ricci}).$          
\end{cor}  

$b)$  $M$ is non-parabolic,  with non-negative Ricci curvature.  Set   
 \begin{equation}\label{Hr}
H(r)=\frac{V(r)}{r^2}\int_{r}^\infty\frac{tdt}{V(t)}, 
 \end{equation} 
 where $V(r)$ denotes  the Riemannian volume of the geodesic ball centered at $o$ with radius $r.$

  \begin{theorem}[=Theorem \ref{main4}]  Let $M$ be a  non-parabolic 
     complete non-compact  K\"ahler manifold   with non-negative  Ricci curvature.  
Let $(N,\omega)$ be a K\"ahler  projective manifold  of  dimension not greater than  that  of $M.$  
 Let $D_1,\cdots,D_q$ be effective divisors in general position on $N$ such that each $D_j$ is cohomologous to  $\omega.$
  Let  $F: M\rightarrow N$ be a  differentiably non-degenerate $\nu$-valued algebroid  mapping. 
    Assume that $q\omega-{\rm Ric}(\omega^n)>0.$
   Then  for any  $\delta>0,$ there exists a subset $E_\delta\subset(1, \infty)$ of finite Lebesgue measure such that 
        \begin{eqnarray*}
&&\left(q-(2\nu-2)\left[\frac{\omega_{FS}|_N}{\omega}\right]\right)T_F(r,\omega)+T_F(r, K_N)+T(r, \mathscr R) \\
&\leq& \sum_{j=1}^q\overline N_F(r,D_j)+O\big(\log T_{F}(r, \omega)+\log H(r)+\delta\log r\big)
         \end{eqnarray*}
    holds for all  $r>1$ outside $E_\delta,$  where $H$ is defined by $(\ref{Hr}).$          
\end{theorem}  

  \begin{cor}[=Corollary \ref{hco2}]\label{}  Let $M$ be a  non-parabolic 
     complete non-compact  K\"ahler manifold   with non-negative  Ricci curvature.  
  Let $H_1,\cdots, H_q$ be hyperplanes   of $\mathbb P^n(\mathbb C)$ in general position.  
  Let  $F: M\rightarrow \mathbb P^n(\mathbb C)$ be a  differentiably non-degenerate $\nu$-valued algebroid  mapping. 
  Assume that $m\geq n.$ Then  for any  $\delta>0,$ there exists a subset $E_\delta\subset(1, \infty)$ of finite Lebesgue measure such that 
     \begin{eqnarray*}
&& \left(q-2\nu-n+1\right)T_F(r,\omega_{FS}) +T(r, \mathscr R) \\
&\leq& \sum_{j=1}^q\overline N_F(r,H_j)+O\big(\log T_{F}(r, \omega_{FS})+\log H(r)+\delta\log r\big)
        \end{eqnarray*}
    holds for all  $r>1$ outside $E_\delta,$  where $H$ is defined by $(\ref{Hr}).$          
\end{cor}  

\section{Algebroid  Mappings}

\subsection{Algebroid  Mappings}~\label{s11}

Let $M$ be a complex  manifold of complex dimension $m.$   Let us recall that 
  a $\nu$-valued algebroid function $W$ on $M$ is defined by an irreducible algebraic equation  of degree $\nu$
  \begin{equation*}\label{eq0}
A_\nu(x) W^\nu+A_{\nu-1}(x) W^{\nu-1}+\cdots+A_0(x)=0, \ \ \  \ A_\nu\not\equiv0
\end{equation*}
with holomorphic functions $A_0, \cdots, A_\nu$  locally defined on $M,$ such that they well define a meromorphic mapping 
$[A_0: \cdots: A_\nu]: M\rightarrow\mathbb P^{\nu}(\mathbb C)$; or in other words,  it is 
 defined     by 
 an irreducible algebraic equation of degree $\nu$
  	 \begin{equation*}\label{eq0}
W^\nu+B_{\nu-1}(x) W^{\nu-1}+\cdots+B_0(x)=0
\end{equation*}
over the  field $\mathscr M(M)$  of all meromorphic functions on $M.$  
Note that     
$$B_j=\frac{A_j}{A_\nu}, \ \ \ \    j=0,\cdots,\nu-1.$$
It is well-known  that algebroid functions extend   the notions   of meromorphic functions and algebraic functions. 

Let us give  any  two   polynomials in $W$: 
\begin{eqnarray*}\label{}
P_1(x, W)&=&W^{\nu_1}+B_{1(\nu_1-1)}(x) W^{\nu_1-1}+\cdots+B_{10}(x),  \label{e1} \\
P_2(x, W)&=&W^{\nu_2}+B_{2(\nu_2-1)}(x) W^{\nu_2-1}+\cdots+B_{20}(x) \label{e2}
\end{eqnarray*}
 over $\mathscr M(M).$ 
 The \emph{resultant}  of $P_1(x,W)$ and $P_2(x,W)$ is defined   by  
$$R=\left | \begin{matrix}
1 & B_{1(\nu_1-1)}   & \cdots & B_{12} & B_{11} & B_{10}&0&0&\cdots&0 \\
0 & 1   & \cdots & B_{13} & B_{12} & B_{11}&B_{10}&0&\cdots&0 \\
\vdots&\vdots& & \vdots&\vdots &\vdots&  \vdots &\vdots & &\vdots \\
0 & 0   & \cdots & 1 & B_{1(\nu_1-1)} & B_{1(\nu_1-2)}&B_{1(\nu_1-3)}&B_{1(\nu_1-4)}&\cdots&B_{10}  \\
1 &B_{2(\nu_2-1)}   & \cdots & B_{22} & B_{21} &0&0&0&\cdots&0 \\
0 & 1   & \cdots & B_{23} & B_{22} &B_{21}&0&0&\cdots&0 \\
\vdots&\vdots& & \vdots&\vdots &\vdots&  \vdots &\vdots & &\vdots \\
0 & 0   & \cdots & 0 & 1 &B_{2(\nu_2-1)}&B_{2(\nu_2-2)}&B_{2(\nu_2-3)}&\cdots&B_{21}  \\
\end{matrix} \right |.$$
It is known that $R$  is  obtained   based on  Sylvester's elimination method. 
  Let $\alpha_1(x),\cdots,\alpha_{\nu_1}(x)$ and  $\beta_1(x),\cdots,\beta_{\nu_2}(x)$ be two groups of   roots of $P_1(x, W)$ and $P_2(x,W),$ respectively. 
Then  
$$R=\prod_{1\leq i\leq \nu_1, 1\leq j\leq \nu_2}(\alpha_i-\beta_j).$$
   Denote by $\mathscr A(M)$  the set of all algebroid functions on $M.$  We have:   
	
	\begin{theorem}\label{field}  Under the usual multiplication and addition, $\mathscr A(M)$ forms a 
	field  with  identity  element $1$ and zero element $0.$
	\end{theorem}
	
	\begin{proof}  We first show that  $\mathscr A(M)$ is closed under the usual multiplication and addition. 
	Let $W_1, W_2\in\mathscr A(M)$ be  defined by irreducible algebraic equations  $P_1(x, W_1), P_2(x, W_2),$ respectively. 
	Set  $W=W_1W_2.$  
	 Substitute  $W_2=W/W_1$ into $P_2(x, W_2),$ we obtain   $P_2(x, W/W_1)=0.$ 
								Regard    $$Q(x, W_1):=W_1^{\nu_2}P_2(x, W/W_1)$$ as a polynomial in $W_1,$ where $\nu_2=\deg P_2.$
						Applying  Sylvester's elimination method, we deduce that  
					$R(x, W)=0,$ where $R$ is  the resultant  of
										 $P_1(x, W_1)$ and  $Q(x, W_1).$ 
It means that $\mathscr A(M)$ is closed under the usual  multiplication. 
	Similarly,   $\mathscr A(M)$ is closed under the usual  addition. 
				Secondly, we prove that $1, 0$ are  identity  element  and zero element of $\mathscr A(M),$
 respectively.  It is clear  that   $1,0\in \mathscr A(M).$  For any $W\in\mathscr A(M)$  with defining  equation 
 $$P(x, W):=W^\nu+B_{\nu-1} W^{\nu-1}+\cdots+B_0=0$$
  over $\mathscr M(M),$ it  suffices to confirm   that $-W\in \mathscr A(M)$ and $1/W\in \mathscr A(M)$  for $W\not\equiv0.$ We see that     $-W$ satisfies  the  algebraic equation 
   $$(-W)^\nu+\tilde B_{\nu-1} (-W)^{\nu-1}+\cdots+\tilde B_0=0,$$
   where  $\tilde B_j=(-1)^{\nu-j}B_j$ with $j=0,\cdots,\nu-1.$ Hence, we have  $-W_1\in \mathscr A(M).$ 
Also,  we have   $1/W\in \mathscr A(M)$ for $W\not\equiv0,$ due to  $W^{-\nu}P(x, 1/W)=0.$
   Finally, the commutative law, associative law and distributive law can  be  examined  also easily. 
 		This completes the proof. 		
	\end{proof}

		\begin{cor}\label{field1}
	 $\mathscr A(M)$ is an algebraically  closed field,  and further an algebraic closure of $\mathscr M(M).$  
	\end{cor}

\begin{proof} At first,  we prove  that  $\mathscr A(M)$ is an algebraically  closed field, i.e.,  any  $W$ defined by 
 an irreducible algebraic equation 
   	 \begin{equation*}
C_{\nu}(x)W^\nu+C_{\nu-1}(x) W^{\nu-1}+\cdots+C_0(x)=0, \ \ \ \  C_\nu\not\equiv0
\end{equation*}
over  $\mathscr A(M)$ is an element of   $\mathscr A(M).$  Using Theorem \ref{field},  the above equation can be written as  
   	 \begin{equation}\label{eqeq00}
W^\nu+\tilde C_{\nu-1}(x) W^{\nu-1}+\cdots+\tilde C_0(x)=0, 
\end{equation}
 where $\tilde C_j=C_j/C_\nu\in \mathscr A(M)$ for $j=0,\cdots,\nu-1.$ Assume that $\tilde C_j$ is defined by  an algebraic equation $P_j(x, \tilde C_j)=0$ over 
$\mathscr M(M)$ for $j=0,\cdots,\nu-1.$
By  Sylvester's elimination for $\nu$ times,  we  can eliminate $\tilde C_0,\cdots, \tilde C_{\nu-1}$ from Eq. $(\ref{eqeq00})$ in turn.   So, we finally get 
$Q(x, W)=0,$ which is an irreducible algebraic equation in $W$  over $\mathscr M(M).$ Hence,  $\mathscr A(M)$ is algebraically  closed. 
Making use of   the definition of  algebroid functions,   we   find  that   $[\mathscr A(M): \mathscr M(M)]$ is an algebraical extension, which deduces that 
 $\mathscr A(M)$ is an algebraic closure of $\mathscr M(M).$
\end{proof}

Let $N$ be a complex manifold of complex dimension $n.$  We shall  introduce the notion   of algebroid mappings as follows. 
	\begin{defi}
 A  multi-valued mapping  $F: M\to N$ is called    an algebroid mapping,
    if  for  every point  $y\in f(x)$ with  all $x\in M,$ 
    $\xi_1\circ F,  \cdots, \xi_n\circ F$
 are   local solutions  of  a  system of $n$ independent  algebraic equations 
 $$P_i\big(x, W_1,\cdots, W_n\big)=0, \ \ \ \   i=1,\cdots, n$$
 in  $W_1,\cdots,W_n$  over $\mathscr M(M),$   where $\xi=(\xi_1,\cdots,\xi_n)$ is a  local holomorphic coordinate around $y.$
\end{defi}

In particular, we treat the case when  $N$ is a complex projective  manifold.      Let  
  $\jmath: N\hookrightarrow\mathbb P^d(\mathbb C)$ be  the  holomorphic  embedding. 
Let  $\zeta=[\zeta_0:\cdots:\zeta_d]$ be      the homogenous coordinate of $\mathbb P^d(\mathbb C).$  Then, we follow an extrinsic  definition of algebroid mappings:  

\begin{defi}\label{defi11}
 Assume that  $N$ is a complex projective  manifold.  A multi-valued mapping $F: M\to N$ is called an algebroid mapping,   if 
$$\jmath\circ F=\big[\zeta_0\circ\jmath\circ F: \cdots: \zeta_d\circ\jmath\circ F\big]:  \   M\to  \mathbb P^d(\mathbb C)$$
 defines  an algebroid mapping, i.e., 
  $$W_i:=\frac{\zeta_i\circ\jmath\circ F}{\zeta_0\circ\jmath\circ F},  \ \ \ \     i=1,\cdots,d$$
 satisfy  a  system of $d$  independent  algebraic equations 
  $$P_i\big(x, W_1,\cdots, W_n\big)=0, \ \ \ \   i=1,\cdots, n$$
 in  $W_1,\cdots,W_n$  over $\mathscr M(M).$
\end{defi}

According to Corollary \ref{field1},  if  $W_1, \cdots, W_d$  are  defined by  some system of  $d$  independent  algebraic equations  over $\mathscr M(M),$  
then each $W_i$ is an algebroid function on $M.$
Thus,  we obtain an alternative, but better extrinsic  definition of algebroid mappings as follows.

\begin{defi}\label{def3}   
 A $\nu$-valued mapping $F: M\to N$ is called a $\nu$-valued algebroid mapping, if 
$$\jmath\circ F=\big[\zeta_0\circ\jmath\circ F:\cdots: \zeta_d\circ\jmath\circ F\big]:  \   M\to \mathbb P^d(\mathbb C)$$
 defines  a $\nu$-valued algebroid mapping, i.e., 
 $$W_i:=\frac{\zeta_i\circ\jmath\circ F}{\zeta_0\circ\jmath\circ F}, \ \ \ \    i=1,\cdots,d$$ 
 are   $\nu_i$-valued algebroid functions on $M,$ respectively,  such that  
 $\nu=\nu_1\cdots\nu_d.$ Here, $W_1,\cdots, W_d$  are  understood to be independent  in choosing  single-valued components  since they are in different coordinate  positions.
 \end{defi}

Combining  Definition \ref{def3} with Theorem \ref{field}, it is immediate  that 

 \begin{cor} Let $W_0, \cdots, W_n$ be     algebroid functions without poles and common zero divisors   on $M.$  Then 
 $$F=\big[W_0:\cdots: W_n\big]:  \  M\to\mathbb P^n(\mathbb C)$$
 is  an algebroid mapping.
  \end{cor}

 \subsection{Quasi-regular  Mapping  Elements}~ 
 
 From now on, we  assume always  that    $N$ is a complex projective manifold.  For  technical reasons,   we  equip   $M$ with  a Hermitian metric. 
   Let $F: M\to N$ be a $\nu$-valued algebroid mapping, i.e.,    $W_i=\zeta_i\circ\jmath\circ F/\zeta_0\circ\jmath\circ F$ is  a $\nu_i$-valued algebroid function  on $M$  with defining  equation 
  \begin{equation}\label{eq}
A_{i\nu_i}(x)W^{\nu_i}_i+A_{i(\nu_i-1)}(x)W^{\nu_i-1}_i+\cdots+A_{i0}(x)=0, \ \ \  \  A_{i\nu_i}\not\equiv0 
\end{equation}
     for   $i=1,\cdots, d.$  Note that    $\nu=\nu_1\cdots\nu_d.$

\begin{defi}
 Let $H$ be a local mapping into $N$ defined in a neighborhood of a point  $x_0\in M.$  A pair $(x_0, H)$ is called a regular mapping element,  if $H$ is holomorphic in a neighborhood of $x_0;$ and called a 
quasi-regular mapping element,  if $H$ is meromorphic in a neighborhood of $x_0.$  
A pair   $(x_0, H)$ is said to be  a mapping element of $F,$ if it is  subordinate to Eq. $(\ref{eq}),$ i.e.,   $W_i\circ\jmath\circ H$ is  a local   solution of Eq. $(\ref{eq})$ in a  neighborhood of $x_0$ for $i=1,\cdots,d.$  
\end{defi}

For  $i=1,\cdots, d,$  put   
\begin{equation}
\Psi_i(x, W_i):=A_{i\nu_i}(x)W^{\nu_i}_i+A_{i(\nu_i-1)}(x)W^{\nu_i-1}_i+\cdots+A_{i0}(x). 
\end{equation}
 Then 
  \begin{eqnarray*}
\Psi_{i, i}(x,W_i)&:=&\frac{\partial\Psi_i(x,W_i)}{\partial W_i} \\
&=&\nu_i A_{i\nu_i}(x)W^{\nu_i-1}_i+(\nu_i-1)A_{i(\nu_i-1)}(x)W^{\nu_i-2}_i+\cdots+A_{i1}(x).
  \end{eqnarray*}

\begin{theorem}\label{thm1} Let $(x_0, y_0)\in M\times N$ be  a pair such   that 
$$\Psi_i(x_0, c_i)=0, \ \ \  \    \Psi_{i, i}(x_0, c_i)\not=0, \ \ \ \     i=1,\cdots,d,$$
where 
$$c_i=\frac{\zeta_i\circ\jmath(y_0)}{\zeta_0\circ\jmath(y_0)}, \ \ \ \ i=1,\cdots,d.$$
Then, there exists a unique regular mapping  element $(x_0, H)$ of $F$  such that  $y_0=H(x_0).$ Moreover, $\zeta_0\circ\jmath\circ H$ has no zeros. 
\end{theorem}
\begin{proof}  
Write 
  \begin{eqnarray*}
&&\Psi_i(x, W_i) \\ 
&=& B_{i\nu_i}(x)(W_i-c_i)^{\nu_i}+B_{i(\nu_i-1)}(x)(W_i-c_i)^{\nu_i-1}+\cdots+B_{i0}(x) \\
&=& (W_i-c_i)P_i(x,W_i)+B_{i0}(x),
  \end{eqnarray*}
where
  \begin{eqnarray*}
&& P_i(x,W_i)\\ 
&=&B_{i\nu_i}(x)(W_i-c_i)^{\nu_i-1}+B_{i(\nu_i-1)}(x)(W_i-c_i)^{\nu_i-2}+\cdots+B_{i1}(x).
  \end{eqnarray*}
By conditions, we have     
  \begin{eqnarray*}
B_{i0}(x_0) &=& \Psi_i(x_0, c_i)=0, \\
P_i(x_0, c_i)&=& \Psi_{i,i}(x_0, c_i):=b_i\not=0.
  \end{eqnarray*}
Using the continuation, there exist $r, \rho>0$ such that 
$$|B_{i0}(x)|<\frac{\rho|b_i|}{3}, \ \ \ \ |P_i(x, W_i)|>\frac{2|b_i|}{3}, \ \ \ \   i=1,\cdots,d$$
for all $x, W_i$ satisfying ${\rm{dist}}(x_0, x)\leq r$ and $|W_i-c_i|\leq\rho.$  In further, restrict  $W_i$ to $|W_i-c_i|=\rho,$  we have for $i=1,\cdots,d$
$$\left|(W_i-c_i)P_i(x,W_i)\right|>\frac{2\rho|b_i|}{3}>\frac{\rho|b_i|}{3}>|B_{i0}(x)|.$$
   According to  Rouch\'e's theorem,   for every   $x$ with   ${\rm{dist}}(x_0, x)\leq r,$ we deduce  that $\Psi_i(x,W_i)=(W_i-c_i)P_i(x,W_i)+B_{i0}$ as well as that $(W_i-c_i)P_i(x,W_i)$  have the same number  of  zeros in the disc $\{|W_i-c_i|<\rho\}.$ In fact,  the number of zeros  is 1,  since    $(W_i-c_i)P_i(x,W_i)$ in $W_i$ has only one zero $W_i=c_i$  in the disc  $\{|W_i-c_i|<\rho\}$ with 
    ${\rm{dist}}(x_0, x)\leq r.$ 
 This implies that  for  $i=1,\cdots,d,$   $\Psi_i(x, W_i)=0$ uniquely determines   a solution  $W_i=u_i(x)$
  in a small  geodesic  ball neighborhood $B(x_0, r_0)$  of $x_0$ with radius $r_0$  such that 
$$\Psi_i(x, u_i)=0, \ \ \  \  c_i=u_i(x_0).$$
Clearly,  every   $u_i$ is continuous on $B(x_0, r_0).$  Again, since $\jmath$ is a holomorphic embedding,   it is easy to verify   that there exists a unique 
 continuous mapping 
$H: B(x_0, r_0)\to N$  such that   
 $\jmath\circ H=[1: u_1:\cdots:u_d]$ with $y_0=H(x_0)$ and $u_i=\zeta_i\circ\jmath\circ H/\zeta_0\circ\jmath\circ H$ for $i=1,\cdots,d.$
  Since  $u_1,\cdots, u_d$ are continuous,  it implies that $\zeta_0\circ\jmath\circ H$ has no zeros. 

Next,  we show that $H$ is holomorphic on $B(x_0, r_0).$
 One just needs to show without loss of generality that $H$ is analytic at $x_0.$
Take  a  local holomorphic coordinate $z=(z_1,\cdots, z_m)$ near  $x_0$ such that   $z(x_0)=\textbf{0}.$ Thanks to  Hartog's theorem, it is sufficient  to show that $u_i$ is holomorphic in every holomorphic direction $z_k$ at $\textbf{0}$ for $i=1,\cdots,d.$   Put  
  \begin{eqnarray*}
u_{i,k}(z_k)&=&u_i(0,\cdots,z_k,\cdots,0), \\
B_{ij,k}(z_k)&=&B_{ij}(0,\cdots, 0, z_k, 0, \cdots,0),
  \end{eqnarray*}
  where  $i=1,\cdots, d,$ $j=0,\cdots, \nu_i,$   and $k=1,\cdots, m.$
Note that 
  \begin{eqnarray*}
&& \frac{u_{i,k}-c_i}{z_k} \\
&=&-\frac{B_{i0, k}}{z_k}\frac{1}{B_{i\nu_i,k}(u_{i,k}-c_j)^{\nu_i-1}+B_{i(\nu_i-1),k}(u_{i,k}-c_i)^{\nu_i2}+\cdots+B_{i1,k}} \\
&\rightarrow& -\frac{1}{B_{i1}(\textbf 0)}\frac{\partial B_{i0}}{\partial z_k}(\textbf 0)
  \end{eqnarray*}
as $z_k\rightarrow0.$ That is, 
$$\frac{\partial u_{i,k}}{\partial z_k}(\textbf 0)=-\frac{1}{B_{i1}(\textbf 0)}\frac{\partial B_{i0}}{\partial z_k}(\textbf 0).$$
Thus, $u_i$ is holomorphic in the holomorphic direction $z_k$ at $\textbf{0}.$ This completes the proof. 
\end{proof}

  \begin{defi}\label{def12}
Let $x_0$ be an arbitrary point in $M.$  We call $x_0$  a critical point of $W_i,$  if either   $A_{i\nu_i}(x_0)=0$ or $\Psi_i(x_0, W_i)$ has a multiple root  as a polynomial in $W_i.$ 
We use  $\mathscr S_{W_i}$ to denote  the set of all critical points of $W_i,$  called the critical set of $W_i.$   
A regular point of $W_i$  means  a point in $\mathscr T_{W_i}:=M\setminus \mathscr S_{W_i},$ 
where $\mathscr T_{W_i}$ is called  the regular set of $W_i.$  
  Denote by $\mathscr M_{W_i}$  the set of all points $x_0$ in $\mathscr S_{W_i}$ 
   such that $\Psi_i(x_0, W_i)$  has a multiple root  as a polynomial in $W_i,$  where $\mathscr M_{W_i}$ is called the multiple set of $W_i$ 
  and a point in $\mathscr M_{W_i}$ is called a multiple point of $W_i.$  
   \end{defi}

For   $x_0\in \mathscr T_{M_i},$  the fundamental theorem of algebra says  that  the  algebraic equation 
$$A_{i\nu_i}(x_0)W^{\nu_i}_i+A_{i(\nu_i-1)}(x_0)W_i^{\nu_i-1}+\cdots+A_{i0}(x_0)=0$$
has exactly $\nu_i$ distinct complex roots. Hence, Theorem \ref{thm1} yields  that 

\begin{cor}\label{cor1} For each   $x\in M\setminus\cup_{i=1}^d\mathscr S_{W_i},$  there exist  exactly $\nu$ distinct regular mapping elements $(x, H_1), \cdots, (x, H_\nu)$ 
of $F.$ Moreover, each $\zeta_0\circ\jmath\circ H_j$ has no zeros. 
\end{cor}

To find  the common roots of $\Psi_i(x,W_i), \Psi_{i,i}(x,W_i),$ we treat  the resultant  $R_{\Psi_i}$ of $\Psi_i(x,W_i), \Psi_{i,i}(x,W_i),$ i.e., 
$$R_{\Psi_i}=\left | \begin{matrix}
A_{i\nu_i} &A_{i(\nu_i-1)}   & \cdots & A_{i2} & A_{i1} &A_{i0}&0&0&\cdots&0 \\
0 &A_{i\nu_i}   & \cdots & A_{i3} & A_{i2} &A_{i1}&A_{i0}&0&\cdots&0 \\
\vdots&\vdots& & \vdots&\vdots &\vdots&  \vdots &\vdots & &\vdots \\
0 & 0   & \cdots & A_{i\nu_i} & A_{i(\nu_i-1)} &A_{i(\nu_i-2)}&A_{i(\nu_i-3)}&A_{i(\nu_i-4)}&\cdots&A_{i0}  \\
B_{i\nu_i} &B_{i(\nu_i-1)}   & \cdots & B_{i2} & B_{i1} &0&0&0&\cdots&0 \\
0 &B_{i\nu_i}   & \cdots & B_{i3} & B_{i2} &B_{i1}&0&0&\cdots&0 \\
\vdots&\vdots& & \vdots&\vdots &\vdots&  \vdots &\vdots & &\vdots \\
0 & 0   & \cdots & 0 & B_{i\nu_i} &B_{i(\nu_i-1)}&B_{i(\nu_i-2)}&B_{i(\nu_i-3)}&\cdots&B_{i1}  \\
\end{matrix} \right |,$$
where $B_{ij}=iA_{ij}$ for $i=1,\cdots,\nu_i.$
Applying   Sylvester's elimination method,   we  see that $R_{\Psi_i}(x)=0$ if and only 
if   $A_{i\nu_i}(x)=0$ or  
 $\Psi_i(x, W_i)$ has a   multiple  root  as a polynomial in $W_i,$ i.e., $x\in\mathscr S_{W_i}.$ 
Therefore,   $\mathscr S_{W_i}$ is an analytic set of complex codimension $1.$
 Define the  \emph{discriminant}  of $\Psi_i(x, W_i)$ by 
$$J_{\Psi_i}(x)=(-1)^{\frac{\nu_i(\nu_i-1)}{2}}\frac{R_{\Psi_i}(x)}{A_{i\nu_i}(x)}\not\equiv0, \ \ \ \  i=1,\cdots,d.$$
Write 
 $$\Psi_i(x, W_i)=A_{i\nu_i}(x)(W_i-w_{i1}(x))\cdots(W_i-w_{i\nu_i}(x)).$$
A direct computation leads to  
$$J_{\Psi_i}(x)=A_{i\nu_i}(x)^{2\nu_i-2}\prod_{1\le j<k\le\nu_i}\big(w_{ij}(x)-w_{ik}(x)\big)^2.$$
\ \ \ \    If $x_0\in \mathscr S_{W_i}\setminus A^{-1}_{i\nu_i}(0),$  then $\Psi_i(x_0, W_i)$ has a multiple root  as a polynomial in $W_i.$  
If $x_0\in A^{-1}_{i\nu_i}(0)\setminus I_{\mathscr A_F},$ then we  set $V_i=1/W_i$ and rewrite (\ref{eq}) as 
$$\Psi_i(x,W_i)=V^{-\nu_i}_i\Phi_i(x,V_i),$$
where 
$$\Phi_i(x,V_i)=A_{i\nu_i}(x)+A_{i(\nu_i-1)}(x)V_i+\cdots+A_{i0}(x)V^{\nu_i}_i.$$
Since   $\Phi_i(x,0)=A_{i\nu_i}(x),$  we can deduce  that $x_0$ is  a zero of $V_i,$ i.e., a pole of $W_i.$  
In fact, it   yields    from Theorem \ref{pole} (see Section \ref{sec2}) that  $J_{\Psi_i}(x_0)=0$ if and only if $x_0\in\mathscr M_{W_i}.$
Thus,  
     Corollary \ref{cor1}  can   extend to  $M\setminus \cup_{i=1}^d\mathscr M_{W_i}.$  
\begin{cor}\label{cor2} For each   $x\in M\setminus \cup_{i=1}^d\mathscr M_{W_i},$  there exist  exactly $\nu$ distinct quasi-regular  mapping elements $(x, H_1), \cdots, (x, H_\nu)$ 
of $F.$ 
\end{cor}

Let $F_1,\cdots, F_\nu$ stand for  the $\nu$ distinct single-valued components of $F.$  
We call $x_0\in M$  a multiple point of  $F,$  if $F_i(x_0)=F_j(x_0)$ for some  $i\not=j.$  Denote by $\mathscr M_{F}$  the set of all multiple points of 
$F,$ which is called the multiple set of $F.$

\begin{lemma}\label{lem22q} We have 
$$\mathscr M_F=\cup_{i=1}^d\mathscr M_{W_i}.$$
\end{lemma}

\begin{proof}  If $x_0\in\mathscr M_{W_k}$ for some $k$ with $1\leq k\leq d,$  then  there  exist  two  distinct single-valued components $W_{ki}, W_{kj}$ of $W_k$ such that $W_{ki}(x_0)=W_{kj}(x_0).$ We may assume without loss of generality that $k=1.$
Let $W_{l1}$ be a single-valued component of $W_l$ for  $l=2,\cdots,\nu.$
Take two  single-valued components $F_i, F_j$ of $F$ such that 
$\jmath\circ F_p=[1: W_{1p}:W_{21}:\cdots:W_{\nu1}]$ with $p=i,j.$
Clearly, we have $\jmath\circ F_i(x_0)=\jmath\circ F_j(x_0).$ Since $\jmath$ is injective, we obtain $F_i(x_0)=F_j(x_0),$ which implies that $x_0\in\mathscr M_F.$
If $x_0\in \mathscr M_F,$ then we have $F_i(x_0)=F_j(x_0)$ for some two distinct single-valued components $F_i, F_j$ of $F.$ It yields  that 
$$\frac{\zeta_k\circ\jmath\circ F_i}{\zeta_0\circ\jmath\circ F_i}=\frac{\zeta_k\circ\jmath\circ F_j}{\zeta_0\circ\jmath\circ F_j}, \ \ \ \ k=1,\cdots,d.$$
This gives  that $x_0\in \cup_{i=1}^d\mathscr M_{W_i}.$ The proof is completed. 
\end{proof}

 By   Corollary \ref{cor2} and Lemma \ref{lem22q}, we have:  

\begin{cor}\label{corrr2} For each   $x\in M\setminus \mathscr M_F,$  there exist  exactly $\nu$ distinct quasi-regular  mapping elements $(x, H_1), \cdots, (x, H_\nu)$ 
of $F.$ 
\end{cor}

 \subsection{Analytic Continuation of Quasi-regular  Mapping  Elements}~\label{sec2}

We review  that a quasi-regular  mapping element of $F$  means  a pair $(x, H)$ such that  $\jmath\circ H$ is  a meromorphic  mapping in a neighborhood $U(x)$ of 
$x.$ 
For  convenience, we  write $$\tilde x=(x, H)=(U(x), H)$$   without  any confusion. 
Instead of $U(x),$      a  geodesic ball neighborhood $B(x, r)$ centered at $x$ with radius $r$ is usually considered.  

 Let $\tilde x_1=(U(x_1),  H_1),$ $\tilde x_2=(U(x_2), H_2)$ be any two  quasi-regular mapping  elements of $F.$
Define an equivalent relation as follows:  
  $\tilde x_1\sim \tilde x_2$   if and only if 
 $$x_1=x_2; \ \ \ \  H_1(x)= H_2(x), 
 \ \ \ \  ^\forall  x\in U(x_1)\cap U(x_2).$$    Set 
 $$\mathcal M^0_F=\big{\{}\text{all quasi-regular  mapping elements of $F$}\big{\}}/\sim.$$ 
  \ \ \ \        A quasi-regular  mapping  element  $\tilde y=(y, G)$ is said to be  
  a  \emph{direct analytic continuation} of  a quasi-regular  mapping  element $\tilde x=(U(x), H),$ if $y\in U(x)$ and $G= H$ in a   neighborhood of $y.$ 
  Using  the  analytic continuation, we   can   regard   that   
  $$\tilde y=(y, G)=(y, H).$$
\ \ \ \    A  neighborhood $V(\tilde x)$ of $\tilde x$ is  defined by a set (containing $\tilde x$) of all  $\tilde y\in \mathcal M^0$ such that   
 $\tilde y$  is a  direct analytic continuation  of $\tilde x$ satisfying  $y\in U_1(x),$ where $U_1(x)$  is a neighborhood of $x$ in $U(x).$
Namely, $V(\tilde x)=\{(y, H): y\in U_1(x)\}.$
In particular,  a $r$-\emph{neighborhood} $V(\tilde x, r)$ of $\tilde x$  (satisfying that $B(x,r)\subset U(x)$), 
is  defined by    
  the set (which contains $\tilde x$) of all  $\tilde y\in \mathcal M^0$ such that   $\tilde y$  is a  direct analytic continuation  of $\tilde x$ satisfying that $y\in B(x, r).$
  In other  words, it is defined  to be   $V(\tilde x, r)=\{(y, H): y\in B(x, r)\}.$
    A subset $\mathcal E\subseteq \mathcal M^0$ is called an open set, if either $\mathcal E$ is an empty set or every   point $\tilde x\in\mathcal E$ is an inner point, i.e., there exists  a $\epsilon$-neighborhood $V(\tilde x, \epsilon)$ of $\tilde x$ such that $V(\tilde x, \epsilon)\subset \mathcal E.$
  It is not hard  to  check that it defines  a topology of $\mathcal M^0.$

  \begin{theorem}\label{thm4} $\mathcal M^0_F$ is a  Hausdorff space. 
  \end{theorem}
  \begin{proof} 
 Pick    $\tilde x_1\not=\tilde x_2\in \mathcal M^0_F.$ If $x_1\not=x_2,$ then we can consider   $\epsilon$-neighborhoods 
 $V(\tilde x_1, \epsilon)$ of $\tilde x_1$  and  $V(\tilde x_2, \epsilon)$ of $\tilde x_2$  with $0<\epsilon<{\rm{dist}}(x_1, x_2)/2.$
 It is clear that         
   $V(\tilde x_1, \epsilon)\cap V(\tilde x_2, \epsilon)=\emptyset.$
       If  $V(\tilde x_1, r_1)\cap V(\tilde x_2, r_2)\not=\emptyset$ for   $x_1=x_2,$ 
      then we can   take a quasi-regular mapping element 
           $(x_0, H_0)\in V(\tilde x_1, \epsilon)\cap V(\tilde x_2, \epsilon),$  
          which  gives     
 $H_1= H_0= H_2$ in some  smaller  neighborhood of $x_0.$ By the uniqueness theorem of analytic mappings, we obtain   $H_1\equiv H_2.$ However,  
   it contradicts with $\tilde x_1\not=\tilde x_2.$
  \end{proof}
 
 \begin{defi}
 Let $\tilde x,$ $\tilde y$ be two  quasi-regular  mapping  elements of $F.$  Let $\gamma: [0, 1]\to M$ be a curve with   $x=\gamma(0)$ and $y=\gamma(1).$ We say that $\tilde y$ is an analytic continuation of $\tilde x$ along $\gamma,$ if 
 
 $(i)$ for any $t\in [0,1],$ there is a  quasi-regular  mapping element $\tilde \gamma(t)$ of $F;$
 
 $(ii)$ for any $t_0\in [0,1]$ and any $\epsilon>0,$ there exists $\delta>0$ such that $|\gamma(t)-\gamma(t_0)|<\epsilon$
  for all $t\in[0,1]\cap\{|t-t_0|<\delta\},$ and that  $\tilde\gamma(t)$ is a direct analytic continuation of $\tilde\gamma(t_0).$
 \end{defi}

 \begin{remark}\label{rek2}  An analytic continuation of  elements in $\mathcal M^0_F$ along a curve $\gamma$ in $M$ is a continuous mapping of $\gamma$ into $\mathcal M^0_F,$ i.e., an analytic continuation of $\tilde x$ to $\tilde y$ along $\gamma$ is a curve  connecting $x$ and $y.$ Moreover, if $\tilde y$ is an analytic continuation of $\tilde x$ along a curve,  then $\tilde x$ is also an analytic continuation
  of $\tilde y$ along the same curve. Thus, the  analytic continuation  is symmetric. 
\end{remark}

\begin{theorem}[Analytic continuation]\label{thm2}
Let  $\gamma: [0, 1]\to M\setminus \mathscr M_F$ be a curve. 
Then, any mapping element $\tilde \gamma(0)\in \mathcal M^0_F$ can  continue analytically to $\gamma(1)$ along $\gamma$ and attain  to   a unique  mapping element $\tilde\gamma(1)\in \mathcal M^0_F.$ 
\end{theorem} 
\begin{proof}
Set
$$\tau=\sup\big{\{}t\in[0,1]:  \ \text{$\tilde\gamma(0)$  can continue  analytically to a $(\gamma(t), H_t)\in\mathcal M^0_F$}\big{\}}.$$
It is clear that  $0<\tau\leq1.$
It is sufficient  to confirm   that  $\tau=1.$ Otherwise, we may assume that $\tau<1.$ Using Corollary \ref{corrr2},   there exist  exactly $\nu$  mapping elements  of $F$
at  $\gamma(\tau),$ saying 
$$\left(B(\gamma(\tau), r_\tau), G_{1}\right), \cdots, \left(B(\gamma(\tau), r_\tau), G_{\nu}\right)\in \mathcal M^0_F.$$
\ \ \   Take $\delta>0$ small enough such that $\{\gamma(t): t\in[\tau-\delta,\tau+\delta]\}\subset B(\gamma(\tau), r_\tau)$ with
 $\delta<\min\{\tau,1-\tau\}.$
Then, there is  a unique integer $k_0\in\{1,\cdots, \nu\}$  such that $H_{\tau-\delta}=G_{k_0}$ 
in a neighborhood of  $\gamma(\tau-\delta).$
Notice that  $G_{k_0}$ is  defined on $B(\gamma(\tau),r_\tau),$  then 
 for any $t\in[\tau-\delta,\tau+\delta],$ one can redefine $(\gamma(t), G_{k_0})\in \mathcal M^0_F.$
   Combined with   
the definition of $\{\gamma(t), H_t\}_{0\leq t\leq\tau-\delta},$ it gives   immediately  
 that $\tilde\gamma(0)$ can   continue analytically  to $\gamma(\tau+\delta)$  along $\gamma,$
 i.e.,  $(\gamma(\tau+\delta), G_{k_0})\in \mathcal M^0_F$ is an analytic continuation of $\tilde\gamma(0)$ along $\gamma.$ However, 
  it contradicts  with the definition  of $\tau.$  Hence, we obtain  $\tau=1.$ The uniqueness of   $\tilde \gamma(1)$ is immediate due to the uniqueness theorem of analytic mappings.  
  \end{proof}
    
\begin{cor}\label{cor3}
Let  $\tilde X=\{(x, H_j)\}_{j=1}^\nu, \tilde Y=\{(y, G_j)\}_{j=1}^\nu\subset\mathcal M^0_F$ be two groups of  distinct  mapping elements  at $x, y\in M\setminus \mathscr M_F,$ respectively.   Then,
 $\tilde X, \tilde Y$ can continue analytically to each other along any curve 
 in $M\setminus \mathscr M_F.$
\end{cor} 
\begin{theorem}[Connectivity]\label{thm3} Any two quasi-regular  mapping elements  in $\mathcal M^0_F$ can continue analytically to each other along a curve in 
$M\setminus \mathscr M_F.$
\end{theorem} 
\begin{proof}
Let $\{(x, H_j)\}_{j=1}^\nu\subset\mathcal M^0_F$  be $\nu$ distinct  quasi-regular  mapping elements at  any point  $x\in M\setminus \mathscr M_F.$ Using  Corollary \ref{cor3},  it  
is sufficient   to prove  that  $(x, H_1), \cdots, (x, H_\nu)$ can  continue analytically to each other.  
Let  $n$ stand for    the largest    integer such that there exist  $n$ mapping elements in  $\{(x, H_j)\}_{j=1}^\nu,$ which can  continue analytically to each other along a curve.
 Clearly,  we have  $1\leq n\leq\nu.$
 Without loss of generality, we shall assume that 
  these  $n$ mapping elements are   $(x, H_1), \cdots, (x, H_n).$ 
For  $j=1,\cdots,n,$  set $$u_{ji}=\zeta_i\circ\jmath\circ H_j, \ \ \ \  i=1,\cdots,d.$$ 
  From the assumption,  $n$ is  also the largest    integer such that the  meromorphic mappings 
  $[1: u_{j1}: \cdots: u_{jd}]$ into $\mathbb P^d(\mathbb C)$ defined in a neighborhood  of $x,$
   with  $j=1,\cdots,n,$
   can  continue analytically to each other along a curve.
  Write 
  $$(W-u_{1i})(W-u_{2i})\cdots(W-u_{ni})=W^n+B_{(n-1)i}W^{n-1}+\cdots+B_{0i},$$
  where  $B_{0i}, \cdots, B_{(n-1)i}$ are holomorphic for $i=1,\cdots,d.$
By Vieta's theorem,  we obtain  
  \begin{eqnarray*}
\sum_{1\leq j\leq n}u_{ji} &=&-B_{(n-1)i} \\
 \cdots\cdots\cdots &&  \\
\sum_{1\leq j_1<\cdots<j_k\leq n}u_{j_1i}\cdots u_{j_ki}&=& (-1)^k B_{(n-k)i} \\
 \cdots\cdots\cdots && \\
u_{1i}\cdots u_{ni}&=& (-1)^n B_{0i} 
  \end{eqnarray*}
  It suffices   to show that
  
   $(i)$  $B_{1i},\cdots, B_{ni}$ can continue analytically to the whole 
$M\setminus \mathscr M_F;$ 

$(ii)$ $n=\nu.$ 

Using Theorem \ref{thm2},  $\{(x, H_j)\}_{j=1}^\nu$ can continue analytically to 
  $\{(y, G_j)\}_{j=1}^\nu$ along a curve for each  $y\in M\setminus \mathscr M_F,$ where $(G_1,\cdots, G_n)$ is a permutation of $(H_1,\cdots, H_n)$ due to the assumption.  
For  $j=1,\cdots,n,$ set $$v_{ji}=\zeta_i\circ\jmath\circ G_j, \ \ \ \  i=1,\cdots,d.$$ 
Then,  $(v_{1i},\cdots, v_{ni})$ is also  a permutation of $(u_{1i},\cdots, u_{ni})$ with $i=1,\cdots,d.$
From     Vieta's theorem,  note  that 
   $v_{1i},\cdots, v_{ni}$ determine
         the same $B_{1i},\cdots, B_{ni}$ in some    
         neighborhood of $y,$ 
          i.e.,  $B_{1i},\cdots, B_{ni}$ can  continue analytically to $y.$ Hence,  $(i)$ holds. 
If $n<\nu,$ then we are able to determine another set $\{B'_{ji}\}_{i,j}$ 
using the rest  mapping elements $\{(x, H_j)\}_j.$
 However, it implies  that 
Eq. (\ref{eq}) is reducible,   a contradiction. 
\end{proof}

According to Remark \ref{rek2} and Theorem \ref{thm3}, we obtain:  

\begin{cor} $\mathcal M^0_F$ is a path-connected space. 
\end{cor}

 \subsection{Properties of   Mapping  Elements near  Multiple  Points}~

Let $(x_0, y_0)\in M\times N$ be  a pair such   that 
$$\Psi_i(x_0, c_i)=0, \ \ \  \    \Psi_{i,i}(x_0, c_i)=0$$
for some  $i\in\{1,\cdots,d\},$ where 
$$c_i=\frac{\zeta_i\circ\jmath(y_0)}{\zeta_0\circ\jmath(y_0)}.$$
  $1^\circ$  $x_0\not\in\cup_{i=1}^d A^{-1}_{i\nu_i}(0)$

 Choosing  a  sufficiently  small neighborhood $U(x_0)$ of $x_0$ such that $A_{i\nu_i}\not=0$ on $\overline{U(x_0)}$ for $i=1,\cdots,d.$
 According to  the continuation of $A_{ij}$ for $1\le i\le d$ and $0\le j\le \nu_i,$   there exists a number $M_0>0$ such that 
 \begin{equation}\label{haoba}
 0\le\min_{\substack{1\le i\le d \\ 0\le j\le\nu_i-1}}\left|\frac{A_{ij}}{A_{i\nu_i}}\right|\leq\max_{\substack{1\le i\le d \\  0\le j\le\nu_i-1}}\left|\frac{A_{ij}}{A_{i\nu_i}}\right|\leq M_{0}, \ \ \ \   ^\forall x\in \overline{U(x_0)}.
 \end{equation}  
According to   Corollary \ref{cor2},  for any  $x, y\in U(x_0)\setminus \mathscr M_F$ with $x\not=y,$ 
  there are   two groups of  distinct quasi-regular mapping  elements 
 $\{x, H_j\}_{j=1}^{\nu}, \{y, G_j\}_{j=1}^{\nu}$ of $F.$ 
     Applying   Theorem \ref{thm2} again,   for a  curve $\gamma: [0,1]\rightarrow U(x_0)\setminus \mathscr M_F$
          with   $\gamma(0)=x$ and $\gamma(1)=y,$   
                    one can extend      $\{x, H_j\}_{j=1}^{\nu}$   analytically to $\{y, G_j\}_{j=1}^{\nu}$ along $\gamma.$
                                            In other words, 
 there exists   a permutation $(i_1,\cdots, i_{\nu})$ of $(1,\cdots,\nu)$ 
   such that $H_1,\cdots, H_{\nu}$ can  be  continued analytically
       to $G_{i_1},\cdots, G_{i_{\nu}}$ along $\gamma,$ respectively.  
   The  symmetric property implies that 
 $\{y, G_j\}_{j=1}^{\nu}$  
  shall   be also  continued  analytically to $\{x, H_j\}_{j=1}^{\nu}$ along $\gamma.$ 
           However,  we  may  mention that  $G_{i_1},\cdots, G_{i_\nu}$ 
                         won't  be, in  general,  continued analytically back to $H_{1},\cdots, H_{\nu}$ along $\gamma,$  respectively.       
        In what follows,  we     consider     the analytic continuation of a mapping element along a  closed Jordan curve surrounding  any multiple point 
      $x_0\in \mathscr M_F.$ 
             Taking    a smooth closed Jordan curve  
   $\gamma_0\subset  U(x_0)$   around $x_0$ such that  which is close to  $x_0$   and  doesn't surround or cross 
        any  connected components  of $\mathscr M_F$ that do not pass through $x_0.$
               Without loss of generality, let us  look at   the  analytic continuation of  $H_1$   along $\gamma_0$ starting from $y_0\in\gamma_0.$  
When $H_1$ returns  to $y_0$  along $\gamma_0,$   it  will  turn   to  a    mapping  
 $H_{k_1}\in\{H_j\}_{j=1}^\nu,$ 
  by which we mean that for a local coordinate chart $(U(x_0), \varphi)$ (here $U(x_0)$ is  small enough), $H_1\circ\varphi^{-1}(z)$ 
    jumps   to a  single-valued component $H_{k_1}\circ\varphi^{-1}(z)$ when $z$ travels  
        around  $\varphi(x_0)$ once  along $\varphi(\gamma_0)$ starting from $\varphi(y_0),$ 
                due to the change of 
 $m$-dimensional argument of $z=(z_1,\cdots,z_m).$
   If $k_1\not= 1,$ then we  will  continue to extend $u_{k_1}$  analytically along $\gamma_0$ in the same direction of rotation. 
Repeat      this action,   we shall   see finally  that  $H_1$ cycles on and on periodically.
Let $\lambda_1$ be    the smallest period  such that $H_1$  can return to itself along all such curves $\gamma_0.$
    Then, one can  obtain        $H_1,  H_{k_1},  \cdots,  H_{k_{\lambda_1-1}}.$  
     We  call $x_0$     a \emph{branch point} with order $\lambda_1-1$ of $H_1$ for  $\lambda_1>1.$ 
                  Repeatedly, it   can be  seen  that $\{H_j\}_{j=1}^\nu$    is  divided  into finitely many    groups. 
                                       Let $l$ be    the number of groups. 
              If the $j$-th 
group has  $\lambda_j$ members for $j=1,\cdots,l,$
 then  $$\lambda_1+\cdots+\lambda_l=\nu.$$
        We  just analyze    the first group:  $H_1, H_{k_1}, \cdots, H_{k_{\lambda_1-1}}.$
 Take    a   suitable   $U(x_0),$    
     there  exists        
   a biholomorphic mapping $\varphi: U(x_0)\to \Delta^m(2)$  such that  
   $$\varphi(x_0)=\textbf{0}, \ \ \ \    \varphi(\gamma_0)=\big\{(0,\cdots,0,z_m)\in \Delta^m(2): \ |z_m|=1\big\},$$
where   $\Delta^m(2)$ is   the polydisc  centered at  $\textbf{0}$ with polyradius $(2,\cdots,2)$ in  $\mathbb C^m.$
Setting   $S_0=\varphi^{-1}(\Delta_m(1)),$ 
where  $\Delta_m(1)$ is  the unit disc centered at $\textbf{0}$  in the $z_m$-plane.  
For  any $x\in S_{0}\cap \gamma_0,$ we have   
 $$H_1(z)=H_1\circ\varphi(x)=H_1(0,\cdots, 0, z_m).$$
 By  Corollary \ref{cor1}, $\zeta_0\circ\jmath\circ H_1$ has no zeros.  Write 
$\jmath\circ H_1=[1: u_{11}:\cdots: u_{1d}],$
where  
$$u_{1i}=\frac{\zeta_i\circ\jmath\circ H_{1}}{\zeta_0\circ\jmath\circ H_{1}}, \ \ \ \   i=1,\cdots,d$$
are   $\lambda_{1i}$-valued algebroid functions,  respectively, with    $\lambda_1=\lambda_{11}\cdots\lambda_{1d}.$ 
Set   
 $$u^*_{1i}(\eta)=u_{1i}(0,\cdots,0,\eta^{\lambda_{1i}}), \ \ \ \    i=1,\cdots,d.$$
 We  show   that $u^*_{1i}$ is a single-valued analytic  function 
    on $\Delta_m(1).$  
  Note  that $\lambda_1$ is the branch order of  $x_0,$ we deduce  that 
 $u^*_{1i}$ is holomorphic on $\Delta_m(1)\setminus\{\textbf{0}\},$ because   
$\varphi(\gamma_0)$ is so close to $\textbf{0}$ that there exist  no any  branch points of $u^*_{1i}$ in $\Delta_m(1)\setminus\{\textbf 0\}.$
 We  only  need to examine  that $\zeta=0$ is a removed singularity of $u^*_{1i}.$ To do so, it suffices   to show    that $u^*_{1i}$ is bounded near $\textbf{0}$
 due to    Riemann's analytic continuation.  
  In fact, 
       it follows  from    
 (\ref{eq}) and (\ref{haoba}) that    
     \begin{eqnarray*}
\left|u_{1i}\right|&=& \frac{1}{\left|A_{i\nu_i} {u_{1i}}^{\nu_i-1}\right|}\left|A_{i(\nu_i-1)}{u_{1i}}^{\nu_i-1}+\cdots+A_{i0}\right| \\
&\leq& M_0\left(1+\frac{1}{|u_{1i}|}+\cdots+\frac{1}{|u_{1i}|^{\nu_{i}-1}}\right)
  \end{eqnarray*}
on $\overline{U(x_0)}.$ If  $|u_{1i}|\geq1,$ then   $|u_{1i}|\leq\nu_i M_0.$  Hence,  it gives 
$$|u_{1i}|\leq\max\{1, \nu_{i} M_0\}, \ \ \ \     ^\forall x\in\overline{U(x_0)}.$$ 
  That is,   $u_{1i}$ is bounded on $\overline{U(x_0)}.$
  It means that  $u^*_{1i}$  is holomorphic on $\Delta_m(1).$  For $i=1,\cdots,d,$
 we  can expand  $u^*_{1i}$  into   the following    Taylor  series   
   \begin{equation}\label{dxj1}
   u^*_{1i}(\eta)=b_{i0}+b_{i\tau_i}\eta^{\tau_i}+b_{i(\tau_i+1)}\eta^{\tau_i+1}+\cdots
   \end{equation}
at   $\eta=0.$  
Let $\mathscr B_{H_1}$ be    the set of all branch points of $H_1$  in $U(x_0).$  
We say  that a branch point $x_0$ of $H_1$  is  \emph{non-singular},  
if $x_0$ does not  lie in any crossings of components of $\mathscr B_{H_1};$ 
and  \emph{singular} otherwise. 
Next, assume   that     $x_0$ is  not  a  singular branch point of $H_1.$
Then,  $\mathscr B_{H_1}$ is   smooth  whenever   $U(x_0)$  is  small enough.
Take a  suitable $U(x_0)$ if necessary.   
Since   $H_1$ is bounded on $\overline{U(x_0)},$ 
and  $\lambda_1$ defines  the branch order of $H_1$ at $x_0$ that  is  
 independent of the choice of $\gamma_0,$  
    we  conclude      from    $(\ref{dxj1})$ and  Riemann’s analytic continuation  that 
 there exists   a biholomorphic mapping $\varphi: U(x_0)\to\Delta^m(1)$ with  
  $$\varphi(x_0)=\textbf{0}; \ \ \ \    \varphi(x)=(\hat z_m(x),0), \ \ \ \
  ^\forall   x\in \mathscr B_{H_1},$$
  such  that 
  each $u_{1i}$  can be expanded  into the  following   Pusieux series 
   $$u_{1i}(z)=B_{i0}(\hat z_m)+B_{i\tau_i}(\hat z_m)z_m^{\frac{\tau_i}{\lambda_1}}+B_{i(\tau_i+1)}(\hat z_m)z_m^{\frac{\tau_i+1}{\lambda_1}}+\cdots$$
   at  $x_0,$ where    $B_{i0}, B_{i\tau_i},\cdots$ are holomorphic functions  in $\hat z_m=(z_1,\cdots,z_{m-1}).$    
   Comparing coefficients, we deduce that   
   $$B_{ij}(\textbf{0})=b_{ij}, \ \ \ \    j=0,\tau_i, \tau_i+1, \cdots$$
   According to the  above arguments,  
     we show that $H_1$ is a $\lambda_1$-valued analytic  mapping defined   in a neighborhood of $x_0.$ 
     We call $(x_0, H_1)$ a \emph{$\lambda_1$-leaf mapping  element}.
    If one cuts $U(x_0)$ along a real hypersurface   passing    through $\mathscr B_{H_1},$ 
  then  $H_1$  can  be separated  into $\lambda_1$ single-valued analytic components, 
    and all these $\lambda_1$  components can turn to each other when they continue analytically along a closed curve around $x_0$ in $U(x_0).$

\noindent $2^\circ$  $x_0\in\cup_{i=1}^d A^{-1}_{i\nu_i}(0)$

Assume that $x_0$ is  not a singular branch point of $H_1.$ 
 For $i=1,\cdots,d,$ we set $W_i=1/V_i$ and define  
      \begin{eqnarray*}
\Phi_i(x,W_i)&:=&W_i^{\nu_i}\Psi_j(x,V_i) \\
&=& A_{i\nu_i}(x)+A_{i(\nu_i-1)}(x)W_i+\cdots+A_{i0}(x)W^{\nu_i}_i.
     \end{eqnarray*}   
 $a)$      $x_0\not\in\cup_{i=1}^d A^{-1}_{i0}(0)$  
 
By the previous arguments, one can also expand $u_{1i}$ locally into a Pusieux
series of the form   $$u_{1i}(z)=C_{i0}(\hat z_m)+C_{i\tau_i}(\hat z_m)z_m^{\frac{\tau_i}{\lambda_1}}+C_{i(\tau_i+1)}(\hat z_m)z_m^{\frac{\tau_i+1}{\lambda_1}}+\cdots$$
   at $x_0.$
If  $C_{i0}\not\equiv0,$  then  $u_{1i}^{-1}$ can be expanded at $x_0$ in the form
     \begin{eqnarray*}
   u_{1i}(z)^{-1}&=& \tilde C_{i(-\tau_i)}(\hat z_m)z_m^{-\frac{\tau_i}{\lambda_1}}+\tilde C_{i(-\tau_i+1)}(\hat z_m)z_m^{-\frac{\tau_i-1}{\lambda_1}}+\cdots \\
  && +~\tilde C_{i(-1)}(\hat z_m)z_m^{-\frac{1}{\lambda_1}}+\frac{1}{C_{i0}(\hat z_m)}+\tilde C_{i1}(\hat z_m)z_m^{\frac{1}{\lambda_1}}+\cdots
        \end{eqnarray*}
   $b)$       $x_0\in\cup_{i=1}^d A^{-1}_{i0}(0)$ 
   
Take a number  $c_0$  with  $\Psi_i(x_0, c_0)\not=0$ for all $i.$  Put $\tilde W_i=W_i-c_0.$ Then 
     \begin{eqnarray*}
\tilde\Psi_i(x,\tilde W_i)&:=&\Psi_i(x,\tilde W_i+c_0) \\
&=& \tilde A_{i\nu_i}(x)\tilde W_i^{\nu_i}+\tilde A_{i(\nu_i-1)}(x)\tilde W_i^{\nu_i-1}+\cdots+\tilde A_{i0}(x).
     \end{eqnarray*}
Evidently,   $\tilde A_{i0}(x_0)=\tilde\Psi_i(x_0, 0)=\Psi_i(x_0, c_0)\not=0.$ Hence,  the case $b)$  turns to the  case $a).$

To conclude,   for each   $x\in M,$ there are   $l=l(x)$ distinct multi-leaf   
mapping elements of $F,$ denoted by  $(x, H_1), \cdots, (x, H_l),$ in which   $H_j$ is a  $\nu_j$-valued analytic    mapping  with     
$1\leq\nu_j\leq\nu$ and $\nu_1+\cdots+\nu_l=\nu.$ 
   Note that   $H_j$ has branch order $\nu_j-1$ at $x$ with $j=1,\cdots,l.$ 
        If  $l<\nu,$  then one says that  $x$  is a branch point with order $\nu-l$ of $F.$
        Let  $\mathscr B_F$  denote   the set of all  branch points of $F,$ which  is an analytic set  in $M,$ called the \emph{branch set} of $F.$ 
         Recall that 
      $$W_i=\frac{\zeta_i\circ\jmath\circ F}{\zeta_0\circ\jmath\circ F}, \ \ \ \ i=1,\cdots,d.$$
      We have shown   that 
\begin{theorem}\label{ses}
Let   $(x_0, H)$ be     
a $\lambda$-leaf  mapping  element of $F.$ 
Assume that 
$$u_{i}=\frac{\zeta_i\circ\jmath\circ H}{\zeta_0\circ\jmath\circ H}, \ \ \ \  i=1\cdots,d$$ 
are   $\lambda_i$-valued   components  of    $W_i$ in a neighborhood of $x_0,$ respectively.   
 Then,  we have $\lambda=\lambda_1\cdots\lambda_d.$   Moreover, 
if $x_0$ is  not a  singular  branch point of $H,$ 
then   $u_i$
 can be expanded into the following  convergent Pusieux series 
   $$u_i(z)=B_{i0}(\hat z_m)+B_{i\tau_i}(\hat z_m)z_m^{\frac{\tau_i}{\lambda_i}}+B_{i(\tau_i+1)}(\hat z_m)z_m^{\frac{\tau_i+1}{\lambda_i}}+\cdots$$
 at $x_0$ in a local holomorphic coordinate $z=(z_1,\cdots,z_m)$  such  that   $z(x_0)$={\rm{\textbf{0}}} and $z(x)=(\hat z_m(x), 0)$ for all $x\in\mathscr B_H,$ 
 where  $\mathscr B_H$ denotes  the branch set of $H$ in a  neighborhood of $x_0,$ and $B_{i0}, B_{i\tau_i}, \cdots$ are  meromorphic functions in $\hat z_m=(z_1,\cdots,z_{m-1}).$ 
\end{theorem}
  
Let   $F_1,\cdots, F_\nu$ be       $\nu$ distinct single-valued  components of $F.$  
Assume that    $(U(x_0), H)$ is    a   $\lambda$-leaf    mapping element  of $F$ at $x_0$ 
      with   distinct  single-valued   components 
 $(U(x_0), H_1), \cdots, (U(x_0), H_\lambda).$ 
Then,    there exists    a permutation 
    $(i_1,\cdots,i_\nu)$ of $(1,\cdots,\nu)$ 
        such that   $H_1=F_1,\cdots, H_\lambda=F_{i_\lambda}$ on $U(x_0).$

 \section{Multi-leaf Complex  Manifolds by  Algebroid  Mappings}

Let $\tilde x=(U(x), H)\in\mathcal M^0_F$ be any mapping element.  A 
 neighborhood $V(\tilde x)$ of $\tilde x$  means  a set of all 
  $\tilde y\in \mathcal M^0_F$ with   $y\in U(x),$ such that   
 $\tilde y$  is a  direct analytic continuation  of $\tilde x,$  i.e.,   
 $V(\tilde x)=\{(y, H): y\in U(x)\}.$  For   a proper   $U(x),$  
  there is  a  biholomorphic mapping $\varphi_x: U(x)\to \mathbb B^m(1)$ 
        such that $\varphi_x(x)=\textbf{0},$ 
        where $\mathbb B^m(1)$ is the unit ball centered at $\textbf{0}$  in $\mathbb C^m.$
         Define a natural mapping $$\psi_{\tilde x}: V(\tilde x)\to \mathbb B^m(1)$$ via  
$$\psi_{\tilde x}(\tilde y)=\varphi_x(y), \ \ \  \  ^\forall \tilde y\in V(\tilde x).$$
 Clearly, $\psi_{\tilde x}$ is a homeomorphism whose  inverse is 
 $$\psi^{-1}_{\tilde x}(z)=\big(\varphi^{-1}_x(z), H\big), \ \ \  \  ^\forall z\in\mathbb B^m(1).$$
\ \ \ \  Set $$\mathscr U_F=\left\{(V(\tilde x), \psi_{\tilde x}): \tilde x\in\mathcal M^0_F\right\}.$$  We  show that  $\mathscr U_F$  is a complex  atlas  of $\mathcal M^0_F.$
Let $V(\tilde x)\cap V(\tilde y)\not=\emptyset.$  It  suffices   to examine     that  the transition mapping 
$$\psi_{\tilde y}\circ\psi^{-1}_{\tilde x}: \ \psi_{\tilde x}\left(V(\tilde x)\cap V(\tilde y)\right)\to \psi_{\tilde y}\left(V(\tilde x)\cap V(\tilde y)\right)$$
is a biholomorphic mapping. This is automatically a homeomorphism. 
Note  that for  any $\tilde a\in V(\tilde x)\cap V(\tilde y),$ we have $H=G$ in a small neighborhood of $a,$ i.e., $(a, H)=(a, G).$ 
A direct  computation leads to  
 $$\psi_{\tilde y}\circ\psi^{-1}_{\tilde x}(z)=\psi_{\tilde y}\big(\varphi^{-1}_x(z), H\big)
 =\psi_{\tilde y}\big(\varphi^{-1}_x(z), G\big)=\varphi_y\circ\varphi^{-1}_x(z)$$
 for  any $z\in \psi_{\tilde x}(V(\tilde x)\cap V(\tilde y)).$ Similarly, we have  
 $$\psi_{\tilde x}\circ\psi^{-1}_{\tilde y}(z)=\varphi_x\circ\varphi^{-1}_y(z), \ \ \ \     ^\forall z\in \psi_{\tilde y}(V(\tilde x)\cap V(\tilde y)).$$ 
 It shows that  $\psi_{\tilde y}\circ\psi^{-1}_{\tilde x}$ is a  biholomorphic mapping. 
 On the other hand,  
  each  $x\in M\setminus \mathscr B_F$  corresponds to  
$\nu$ distinct  mapping  elements
  $\{(x, H_j)\}_{j=1}^\nu\subset\mathcal M^0_F.$  That is,  
 we conclude that
\begin{theorem} $\mathcal M^0_F$ is  a non-compact $\nu$-leaf complex manifold of complex dimension $m.$
\end{theorem}
Set 
 $$\mathcal M^1_F=\big{\{}\text{all regular mapping elements of $F$}\big{\}}/\sim.$$ 
\begin{cor} $\mathcal M^1_F$ is  a non-compact $\nu$-leaf complex manifold of complex dimension $m,$ which is a submanifold of $\mathcal M^0_F.$
\end{cor}

More generally,  we  would    construct a $\nu$-leaf  complex manifold determined  by  all  the mapping  elements of $F.$ 
Let  $\tilde x_1=(U(x_1), H_1), \tilde x_2=(U(x_2), H_2)$ be any  two   mapping  elements of $F.$ 
We define an  equivalent relation as follows:  $\tilde x_1\sim\tilde x_2$ if and only if 
 $$x_1=x_2; \ \ \ \  H_{10}(x)= H_{20}(x), \ \  \ \    ^\forall x\in  \left(U(x_1)\setminus\mathscr B_{H_1}\right)\cap \left(U(x_2)\setminus\mathscr B_{H_2}\right)$$
 for some   single-valued component $H_{j0}$ of $H_j$ with    $j=1,2,$  where $\mathscr B_{H_j}$ is the branch set of $H_j$ for   $j=1,2.$ 
 Set 
 $$\mathcal M_F=\big{\{}\text{all mapping elements of $F$}\big{\}}/\sim.$$

\begin{defi}
A  quasi-regular  mapping  element $(y, G)$ is called   a direct analytic continuation of a multi-leaf   mapping element $(U(x), H),$ if 
both $y\in U(x)\setminus\mathscr B_{H}$ and $G=H_{0}$
in a  neighborhood of $y$ for some single-valued component $H_{0}$ of $H,$ where $\mathscr B_H$ is the branch set of $H.$   
 If so, one may regard that $(y, G)=(y, H_0).$
\end{defi}
 Let $\tilde x=(U(x), H)\in\mathcal M_F.$  A  neighborhood $V(\tilde x)$ of $\tilde x$ is  defined to be the set that  contains the following elements: 
      \begin{enumerate}
   \item[$\bullet$] all  $\tilde y\in \mathcal M_F$ with  $y\in U_1(x)\setminus\mathscr B_{H}$ satisfying  that   
 $\tilde y$  is a  direct analytic continuation  of $\tilde x$; 
  \item[$\bullet$]  all  $(y, H)\in \mathcal M_F$ with  $y\in U_1(x)\cap \mathscr B_H.$ In the above,    $U_1(x)\subset U(x)$ is a neighborhood of $x.$
    \end{enumerate}
  
In particular,  a $r$-\emph{neighborhood} $V(\tilde x, \rho)$ of $\tilde x$ with $B(x, r)\subset U(x)$ 
is  defined to be  the set that contains the following elements: 
  
        \begin{enumerate}
   \item[$\bullet$] 
    all  $\tilde y\in \mathcal M_F$ with  $y\in B(x, r)\setminus\mathscr B_H$ satisfying   that   
 $\tilde y$  is a  direct analytic continuation  of $\tilde x;$ 
  \item[$\bullet$]  all  $(y, H)\in \mathcal M_F$  with $y\in B(x, r)\cap \mathscr B_H.$ 
    \end{enumerate}
      
     A subset $\mathcal E\subset \mathcal M_F$ is called an open set, if  $\mathcal E$ is an 
     empty set or each  point $\tilde x\in\mathcal E$ is an inner point, i.e., there exists  a $\epsilon$-neighborhood $V(\tilde x, \epsilon)$ of $\tilde x$ such that  
       $V(\tilde x, \epsilon)\subset \mathcal E.$
    It gives   a topology of $\mathcal M_F,$
 which    contains  the topology of $\mathcal M^0_F.$
   If  $\tilde x$ is a $\lambda$-leaf mapping element,   
    then   $V(\tilde x)$ has  $\lambda$ leaves at $\tilde x.$
  We call $V(\tilde x)$  a \emph{$\lambda$-leaf neighborhood} of $\tilde x.$

  \begin{theorem} $\mathcal M_F$ is a connected  Hausdorff space. 
  \end{theorem}
  \begin{proof} 
 The  Hausdorff property of $\mathcal M_F$  is proved  using the similar arguments as in the proof of Theorem \ref{thm4}. We show   
   that $\mathcal M_F$ is connected.  Otherwise,   we may assume that  $\mathcal M_F=\mathcal A\cup\mathcal B,$ where $\mathcal A, \mathcal B$ 
   are two disjoint open subsets of $\mathcal M_F.$
      Set   $\mathcal S=\{\tilde x\in \mathcal M_F: x\in\mathscr B_F\}.$ It is clear  that     
          $\mathcal M^0_F=(\mathcal A\setminus \mathcal S)\cup (\mathcal B\setminus\mathcal S).$ 
                 Since $\mathscr B_F$ is an analytic set of complex codimension $1,$ 
                                  we see that 
                              $\mathcal A\setminus \mathcal S, \mathcal B\setminus\mathcal S$ are   two disjoint open subsets of $\mathcal M^0_F.$ 
                               This 
        implies that $\mathcal M^0_F$ is not connected, which is a contradiction. 
  \end{proof}

  Next, we   endow     $\mathcal M_F$ with a complex structure such that $\mathcal M_F$ is a complex manifold.  Let    $\tilde x=(U(x), H)\in \mathcal M_F$ be any $\lambda$-leaf mapping element,  in which  $x$ is not a singular branch point of $H.$ Assume that 
$$u_{i}=\frac{\zeta_i\circ\jmath\circ H}{\zeta_0\circ\jmath\circ H}, \ \ \ \    i=1,\cdots,d$$
are   $\lambda_i$-leaf function elements of $W_i$  on $U(x),$ respectively. 
        By   Theorem \ref{ses},   we have  $\lambda=\lambda_1\cdots\lambda_d.$  In further,  choosing   $U(x)$  properly if necessary,  
there  
    exists     a biholomorphic mapping
     $\varphi_x: U(x)\to\Delta^m(1)$ such that        
  every  $u_i$  can be expanded into  the convergent Pusieux series 
        \begin{eqnarray*}
       u_i(z)&=&u_i\circ\varphi^{-1}_x(z) \\
       &=& B_{i0}(\hat z_m)+B_{i\tau}(\hat z_m)z_m^{\frac{\tau_i}{\lambda_i}}+B_{i(\tau_i+1)}(\hat z_m)z_m^{\frac{\tau_i+1}{\lambda_i}}+\cdots
             \end{eqnarray*}
on $U(x)$ with      $\varphi_x(x)=\textbf{0}$ and   $\varphi_x(y)=(\hat z_m(y), 0)$
 for  all $y\in\mathscr B_H.$  
 
Let  $V(\tilde x)$  be  a $\lambda$-leaf neighborhood  of $\tilde x$ containing 
 the following elements: 
       \begin{enumerate}
   \item[$\bullet$]  all  $\tilde y\in \mathcal M_F$  with $y\in U(x)\setminus\mathscr B_H$ satisfying  that   
 $\tilde y$ is a  direct analytic continuation  of $\tilde x;$   
   \item[$\bullet$]  
    all  $(y, H)\in \mathcal M_F$ with  $y\in U(x)\cap \mathscr B_H.$
      \end{enumerate}
       Let $H_1,\cdots, H_\lambda$ denote  $\lambda$ distinct single-valued  components of $H.$ 
Define a mapping $\psi_{\tilde x}:  V(\tilde x)\to \Delta^m(1)$
 by 
      \begin{eqnarray*}
\psi_{\tilde x}(y, H)&=& \big(\hat z_m(y), 0\big),  \quad  \quad \quad \quad  \quad \quad \quad \quad \quad \quad  \quad \ \    ^\forall y\in U(x)\cap\mathscr B_H; \\
   \psi_{\tilde x}(y, H_j)&=&\Big(\hat z_m(y), \ |z_m(y)|^{\frac{1}{\lambda}}e^{i\frac{\arg z_m(y)+2(j-1)\pi}{\lambda}}\Big), 
 \quad  ^\forall y\in U(x)\setminus\mathscr B_H
     \end{eqnarray*}
for $j=1,\cdots,\lambda,$ where $z(y)=\varphi_x(y).$ We note that $\psi_{\tilde x}$  is a homeomorphism and its inverse   is  
$$
\psi^{-1}_{\tilde x}(z) = 
\begin{cases}
\big(\varphi^{-1}_x(z), H\big), & ^\forall  z_m=0; \\
\left(\varphi_x^{-1}\big(\hat z_m, |z_m|^\lambda e^{i\lambda\arg z_m}\big), H_j\right),  &  ^\forall z_m\not= 0,
\end{cases}$$
where $j$ is a positive integer such that $\big(2(j-1)\pi-\lambda\arg z_m\big)\in[0,2\pi).$
Denote by $\mathcal M_{F, \rm{sing}}$  the set of all  mapping elements $(x, H)\in\mathcal M_F$ satisfying  that $x$
 is a singular branching point  of $H,$ 
 and  by  $\mathscr B_{F, \rm sing}$  the set of all corresponding  singular branching points. 
 Note that  $\mathscr B_{F, \rm sing}$ is an analytic set of complex codimension not smaller   than $2.$
Set $$\mathscr U=\big\{(V(\tilde x), \psi_{\tilde x}): \tilde x\in\mathcal M^*_F\big\}, \ \ \ \  \mathcal M^*_F=\mathcal M_F\setminus\mathcal M_{F, \rm{sing}}.$$
 We  prove  that  $\mathscr U$  is a complex atlas  of $\mathcal M^*_F.$
Assume that   $V(\tilde x)\cap V(\tilde y)\not=\emptyset.$  It is sufficient  to verify that   the transition mapping 
$$\psi_{\tilde y}\circ\psi^{-1}_{\tilde x}: \ \psi_{\tilde x}\big(V(\tilde x)\cap V(\tilde y)\big)\to \psi_{\tilde y}\big(V(\tilde x)\cap V(\tilde y)\big)$$
is a biholomorphic mapping.  It is automatically a homeomorphism. 
Through  a simple  analysis,  we only  need to deal with  the  situation  in which   $\tilde x=(x, H)$ is a $\lambda$-leaf  quasi-regular   mapping element 
and $\tilde y=(y, G)$ is  a  regular  mapping   element.  For  any  $\tilde a\in V(\tilde x)\cap V(\tilde y),$ 
there exists a single-valued component $H_j$ of $H$ such that $H_j=G$ in some    neighborhood of $a,$ i.e., $(a, H_j)=(a, G).$ It is clear that 
$z_m\not=0,$ because  all  mapping elements in $V(\tilde y)$ are quasi-regular. 
 A direct computation leads to 
      \begin{eqnarray*}
\psi_{\tilde y}\circ\psi^{-1}_{\tilde x}(z)&=&\psi_{\tilde y}\left(\varphi_x^{-1}\big(\hat z_m, |z_m|^\lambda e^{i\lambda\arg z_m}\big), H_j\right) \\
&=&\psi_{\tilde y}\left(\varphi_x^{-1}\big(\hat z_m, |z_m|^\lambda e^{i\lambda\arg z_m}\big), G\right) \\
&=&\varphi_y\circ\varphi^{-1}_x\Big(\hat z_m, |z_m|^\lambda e^{i\lambda\arg z_m}\Big) \\
&=& \varphi_y\circ\varphi^{-1}_x\big(\hat z_m, z_m^\lambda\big),
      \end{eqnarray*}
where $j$ is a positive integer such that $\big(2(j-1)\pi-\lambda\arg z_m\big)\in[0,2\pi).$ Hence,  $\psi_{\tilde y}\circ\psi^{-1}_{\tilde x}$ is a holomorphic mapping. On the other hand, its inverse is 
      \begin{eqnarray*}
\psi_{\tilde x}\circ\psi^{-1}_{\tilde y}(z)&=&\psi_{\tilde x}\big(\varphi_y^{-1}(z), G\big) \\
&=& \psi_{\tilde x}\big(\varphi_y^{-1}(z), H_j\big) \\
&=& \varphi_x\circ\varphi^{-1}_y\Big(\hat z_m, \ |z_m|^{\frac{1}{\lambda}}e^{i\frac{\arg z_m+2(j-1)\pi}{\lambda}}\Big) \\
&=&\varphi_x\circ\varphi^{-1}_y\Big(\hat z_m,  \big(z_m^{\frac{1}{\lambda}}\big)_j\Big),
      \end{eqnarray*}
    where $\big(z_m^{1/\lambda}\big)_j$ is the $j$-th single-valued component of $z_m^{1/\lambda}.$ Evidently,   $\psi_{\tilde x}\circ\psi^{-1}_{\tilde y}$ is also a holomorphic mapping. Hence, $\mathcal M^*_F$ is a  non-compact $\nu$-leaf complex manifold of complex dimension $m.$ 
    
   We finally treat   the singular branch  set $\mathcal M_{F, \rm sing}.$  Let $\pi: \mathcal M_F\to M$ be the natural projection which sends   any  mapping  element 
    $(x, H)$ to $x.$ Note that   $\pi: \mathcal M^*_F \to M\setminus\mathscr B_{F, \rm sing}$ is a  $\nu$-sheeted ramified analytic   covering such that  
    $$M=\pi(\mathcal M^*_F)\cup\pi(\mathcal M_{F, \rm sing})=\pi(\mathcal M^*_F)\cup \mathscr B_{F, \rm{sing}}.$$
         Since  $\mathscr B_{F, {\rm sing}}$ is an analytic set of complex codimension not smaller  than $2,$  the complex structure of $M$  is  a natural extension of the complex structure of $\pi(\mathcal M^*_F)=M\setminus\mathscr B_{F, \rm sing}.$ It implies   that  
             we  can naturally extend the complex structure of $\mathcal M^*_F$   to $\mathcal M_{F, \rm sing}$ via $\pi.$ 
             
                  Therefore,    we have concluded that  
   \begin{theorem}\label{cover}
    $\mathcal M_F$ is a  $\nu$-sheeted ramified analytic covering of $M.$ We call $\mathcal M_F$ the $\nu$-leaf complex manifold by $F.$
      \end{theorem} 
   It is clear that    the above theorem implies that 
 \begin{cor}[Uniformization]\label{uniform}
 An algebroid mapping $F: M\to N$ can be lifted to a meromorphic mapping $f:  \mathcal M_F\to N$ via the natural projection $\pi: \mathcal M_F\to M$ by mapping  $(x, H)$ to $x$ 
such that $F=f\circ\pi^{-1}.$
\end{cor}

\section{Estimate of  Branch Divisors}
   
   Let $W_i=\{W_{ij}\}_{j=1}^{\nu_i}$ be the  $\nu_i$-valued algebroid function on $M$ defined by Eq. $(\ref{eq})$ for    $i=1,\cdots,d.$ 
   
   We need the following lemma: 
         \begin{lemma}\label{pole} For $i=1,\cdots,d,$ we have 
   $$\left(W_i=\infty\right)=(W_{i1}=\infty)+\cdots+(W_{i\nu_i}=\infty)=(A_{i\nu_i}=0).$$
   \end{lemma}
   \begin{proof} Let $x_0$ be an arbitrary  zero of $A_{i\nu_i}.$
   Since  the indeterminacy locus $I_{\mathscr A_i}$ of 
     $$\mathscr A_{i}=[A_{i0}: \cdots: A_{i\nu_i}]: \ M\rightarrow\mathbb P^{\nu_i}(\mathbb C)$$
 is either  empty  or  of complex codimension 2,  
       we obtain    $I_{\mathscr A_i}=0$ in the sense of divisors. 
   Thus,  we  may assume that $x_0\not\in I_{\mathscr A_i}.$  It follows  that    
   \begin{equation}\label{ppt} 
   n\big(x_0, A_{i\nu_i}=0\big)=\max_{0\leq j\leq\nu_i-1}n\Big(x_0, \frac{A_{ij}}{A_{i\nu_i}}=\infty\Big),
   \end{equation}
where $n(x_0, g=a)$ denotes the number of  $a$-valued points of  a function $g$ at $x_0.$ 
Without loss of generality,   we  assume  that  $W_{i1}, \cdots, W_{ip}$ are all  those components  that  take $\infty$ at $x_0,$ 
According to  Vieta's theorem,  we obtain    
  \begin{eqnarray*}
\sum_{1\leq j\leq \nu_i}W_{ij} &=&-\frac{A_{i(\nu_i-1)}}{A_{\nu_i}} \\
 \cdots\cdots\cdots && \\
\sum_{1\leq j_1<\cdots<j_p\leq \nu_i}W_{ij_1}\cdots W_{ij_p}&=& (-1)^p\frac{A_{i(\nu_i-p)}}{A_{i\nu_i}} \\
 \cdots\cdots\cdots &&  \\
W_{i1}\cdots W_{i\nu_i}&=& (-1)^{\nu_i}\frac{A_{i0}}{A_{i\nu_i}} 
  \end{eqnarray*}
    which implies that   $A_{i(\nu_i-p)}(x_0)\not=0,$ because  $w_{i1}\cdots w_{ip}$  reaches   the maximal  pole order  at  $x_0\not\in I_{\mathscr A_i}.$
         It yields from (\ref{ppt}) that   
        \begin{eqnarray*}
  &&  n\big(x_0, W_i=\infty\big) \\
  &=&n(x_0, W_{i1}=\infty)+\cdots+n(x_0, W_{ip}=\infty) \\
       &=& n(x_0, W_{i1}\cdots W_{ip}=\infty) \\
       &=& n\Big(x_0, \frac{A_{i(\nu_i-p)}}{A_{\nu_i}}=\infty\Big) \\
       &=& n\big(x_0, A_{i\nu_i}=0\big), 
      \end{eqnarray*}
which  leads to  that 
   $(W_i=\infty)=(A_{i\nu_i}=0).$  
\end{proof}

Setting  $V_i=1/W_i$  and using  the  similar arguments, we can also show  the following  corollary. 

   \begin{cor}\label{zero}  For $i=1,\cdots,d,$ we have 
   $$\left(W_i=0\right)=(W_{i1}=0)+\cdots+(W_{i\nu_i}=0)=(A_{i0}=0).$$
   \end{cor}

For any  $x\in \mathscr B_F,$   we denote by   $(x, H_1), \cdots, (x, H_l)$   the all  $l=l(x)$ distinct    
mapping elements of $F$  at $x$ with $\lambda_1,\cdots,\lambda_l$ leaves,  respectively.
Recall     that $\lambda_1+\cdots+\lambda_l=\nu$ and $x$ is a $(\nu-l)$-order branching point of $F.$ 
The \emph{branching divisor}  of $F$  is  a  divisor  $\mathscr D_F$ on $M$  such  that    
 $${\rm{Supp}}\mathscr D_F=\mathscr B_F, \ \ \ \   {\rm{Ord}}_{x}\mathscr D_F=\nu-l(x), \ \ \  \  ^\forall x\in\mathscr B_F.$$  
\ \ \ \   Let $(x_0, H)$ be any $\lambda$-leaf mapping element of $F$ at $x_0.$
One may assume without loss  of generality  that    $\zeta_0\circ\jmath\circ H\not\equiv0.$ 
Assume that 
$$u_i=\frac{\zeta_i\circ\jmath\circ H}{\zeta_0\circ\jmath\circ H}, \ \ \ \   i=1,\cdots,d$$
are $\lambda_i$-valued components of $W_i$ at $x_0,$ respectively, i.e., $(x_0, u_i)$ are $\mu_i$-leaf function elements of 
$W_i$ at $x_0,$ respectively. Note that 
\begin{equation}\label{opper}
\mu_1\cdots\mu_d=\lambda.
\end{equation}
Write   $u_i=\{u_{ij}\}_{j=1}^{\mu_i}$ for $i=1,\cdots,d.$ Then, there are $\mu_i$ distinct single-valued components of $W_i,$ saying  
without loss of generality that $W_{i1}, \cdots, W_{i\mu_i},$ such that $u_{ij}=W_{ij}$ for  $j=1,\cdots,\mu_i$ in some neighborhood of $x_0.$
  Let  $F_1,\cdots, F_\nu$ be the  $\nu$ distinct single-valued  components of $F.$ Then 
$$\jmath\circ F_i:=\big[f_{i0}:\cdots:f_{id}\big], \ \ \ \  i=1,\cdots,\nu$$
 are  $\nu$ distinct single-valued  components of $\jmath\circ F.$
 For  $i=1,\cdots,\nu,$ put 
$$g_{ij}=\frac{f_{ij}}{f_{i0}},  \ \ \ \   j=1,\cdots, d.$$
It is clear  that  each $g_{ik}$ is a single-valued component of $W_k$ with $k=1,\cdots,d.$ By (\ref{opper}), it is not hard to deduce that  
\begin{equation}\label{keyfor}
\Big\{\underbrace{W_{i1},\cdots, W_{i1}}_{\lambda/\mu_i}; \cdots; \underbrace{W_{id},\cdots, W_{id}}_{\lambda/\mu_i}\Big\}\subset\big\{g_{1k},\cdots, g_{\nu k}\big\}.
\end{equation}
That is, there are  $\lambda/\mu_i$ copies of $W_{ij}$ among  $g_{1k},\cdots, g_{\nu k}$ for $j=1,\cdots,\mu_i.$ 

Set 
\begin{equation}\label{JJ}
J_k=\big(f_{10}\cdots f_{\nu0}\big)^{2\nu-2}\prod_{1\leq i<j\leq\nu}\big(g_{ik}-g_{jk}\big)^2, \ \ \ \     k=1,\cdots,d.
\end{equation}

 \begin{theorem}\label{esti}  For $k=1,\cdots, d,$ we have  
$$\mathscr D_F\leq \big(J_k=0\big).$$
 \end{theorem}
 
\begin{proof}   For  any $x_0\in\mathscr B_F,$ we denote by    $(x_0, H_1),\cdots, (x_0,  H_{l})$    the all  distinct  mapping  elements  of $F$ at $x_0,$ with 
$\lambda_1, \cdots,  \lambda_{l}$ leaves,  respectively. 
Note that   $\lambda_1+\cdots+\lambda_{l}=\nu$ as well as     
\begin{equation}\label{hayy}
{\rm{Ord}}_{x_0}\mathscr D=\sum_{s=1}^l(\lambda_s-1)=\nu-l.
\end{equation}
Assume that  for $s=1,\cdots,l$   
$$u^{s}_{i}=\frac{\zeta_i\circ\jmath\circ H_s}{\zeta_0\circ\jmath\circ H_s}, \ \ \ \   i=1,\cdots,d$$
are $\mu^s_{i}$-valued components of $W_i$ at $x_0,$ respectively, i.e., $(x_0, u^s_i)$ are $\mu^s_{i}$-leaf function elements of 
$W_i$ at $x_0,$ respectively. Note that 
\begin{equation*}
\mu^s_{1}\cdots\mu^s_{d}=\lambda_s, \ \ \ \ s=1,\cdots,l.
\end{equation*}
We write   $u^s_i=\{u^s_{ij}\}_{j=1}^{\mu^s_i}$ for all $s, i.$ Then, there exist  $\mu^s_i$ distinct single-valued components of $W_i,$ saying  
without loss of generality that $W_{i1}, \cdots, W_{i\mu^s_i},$ such that $u^s_{ij}=W_{ij}$ for  $j=1,\cdots,\mu^s_i$ in some neighborhood of $x_0.$
Since the set of all singular branch points of $F$ has complex codimension  not smaller  than $2,$ 
  it is sufficient  to deal with     the case  where   $x_0$ is a  non-singular branch point of $F.$  
By  Theorem \ref{ses}, 
we  may write    for all $s, i$
\begin{equation}\label{termss}
   u^s_i(z)=B^s_{i0}(\hat z_m)+B^s_{i\tau^s_i}(\hat z_m)z_m^{\frac{\tau^s_i}{\mu^s_i}}+B^s_{i(\tau^s_i+1)}(\hat z_m)z_m^{\frac{\tau^s_i+1}{\mu^s_i}}+\cdots
   \end{equation}
 in a local holomorphic coordinate $z=(\hat z_m, z_m)$ around  $x_0,$  such that $x_0$ has a local holomorphic coordinate $z=\textbf{0}.$  

Take any $s\in\{1,\cdots,l\}.$ If   $x_0$ is not a pole of $u^s_i,$  then we have $B^s_{i0}(\textbf{0})\not=\infty.$   Using     (\ref{keyfor}) and  (\ref{termss}),   we deduce  that 
there exist  at least 
$\lambda_s(\lambda_s-1)/2$ terms  which   contain    the factor $z_m^{2/\mu^s_i}$   in 
$J_k.$ Since $\mu^s_i|\lambda_s,$  there exist   at least $\lambda_s-1$ terms  which  contain  the factor $z_m$ in  $J_k.$
If    $x_0$ is  a pole of $u^s_i,$ then it is  also  a  pole of $W_{i1}\cdots W_{i \mu^s_i}.$
 According to  Lemma \ref{pole} and (\ref{keyfor}), we are led to   
$${\rm{Ord}}_{x_0}\big(u^s_i=\infty\big)\leq {\rm{Ord}}_{x_0}\big(g_{1k}\cdots g_{\nu k}=\infty\big)
= {\rm{Ord}}_{x_0}\big(f_{10}\cdots f_{\nu0}=0\big).$$
Combining the  both cases and applying  (\ref{hayy}), we conclude that  
$${\rm{Ord}}_{x_0}\mathscr D_F\leq {\rm{Ord}}_{x_0}\big(J_k=0\big).$$ 
This completes the proof. 
\end{proof}

\section{Equidistribution of Meromorphic Mappings}

Let $(M, g)$ be an $m$-dimensional complete non-compact K\"ahler manifold  with
 Laplace-Beltrami operator $\Delta$ and 
 K\"ahler form
$$\alpha=\frac{\sqrt{-1}}{\pi}\sum_{i,j=1}^mg_{i\bar j}dz_i\wedge d\bar z_j.$$
\ \ \ \   Given  a family $\{\Delta(r)\}_{r>0}$ of   precompact domains $\Delta(r)\subset M$ with   smooth boundaries $\partial\Delta(r)$ which  can exhaust  $M.$
       Fix a reference  point $o\in \Delta(r).$ We use  $g_{r}(o, x)$   to stand for  
              the  Green function of $\Delta/2$ for $\Delta(r)$ with a pole at $o$ satisfying Dirichlet boundary condition, i.e., 
                $$-\frac{1}{2}\Delta g_{r}(o, x)=\delta_o(x), \ \ \ \  ^\forall x\in\Delta(r); \ \ \ \  \  g_{r}(o, x)=0, \ \ \ \ ^\forall x\in\partial\Delta(r),$$
  where $\delta_o$ is the Dirac's delta function  on $M$ with a pole at $o.$  Note that the harmonic measure $\pi_{r}$ on $\partial\Delta(r)$ with respect to $o$ is defined by 
$$d\pi_{r}(x)=\frac{1}{2}\frac{\partial  g_{r}(o,x)}{\partial \vec \nu}d\sigma_{r}(x),$$
where $d\sigma_{r}$ is the induced Riemannian volume element of $\partial\Delta(r)$ and $\partial/\partial \vec{\nu}$ is the inward normal derivative on $\partial\Delta(r).$ 
Moreover,  let  $(N, \omega)$ be a compact K\"ahler manifold of complex dimension $n.$

\subsection{First Main Theorem}~
  
We introduce   Nevanlinna's functions.   Let $f: M\to N$ be a meromorphic mapping. 
The characteristic function of $f$ with respect to $\omega$ is defined by
  \begin{eqnarray*}
  T_f(r, \omega)&=& \frac{\pi^m}{(m-1)!}\int_{\Delta(r)}g_{r}(o,x)f^*\omega\wedge\alpha^{m-1} 
    \end{eqnarray*}
Since $N$ is compact,   for another K\"ahler metric $\omega'$ on $N,$ we have
$$T_f(r,\omega')=T_f(r,\omega)+O(1).$$ 
\ \ \ \   Let $D$ be an effective  divisor   cohomologous to $\omega,$ by which we mean that there exists  a  function $u_D\geq0$ on $N$ such that 
\begin{equation}\label{888}
\omega-[D]=2dd^cu_D
\end{equation}
in the sense of currents, here 
$d=\partial+\overline{\partial}$ and $d^c=(\sqrt{-1}/4\pi)(\overline{\partial}-\partial).$ 
Thus, $-u_D$ is quasi-plurisubharmonic, i.e., $dd^c(-u_D)\geq-\omega/2.$
  The proximity  function of $f$   with respect to $D$ is defined by
$$  m_f(r, D)=\int_{\partial\Delta(r)} u_D\circ fd\pi_{r}.$$
Define  the counting  function of $f$ with respect to $D$ by
 \begin{eqnarray*}
  N_f(r, D)&=&\frac{\pi^m}{(m-1)!}\int_{f^*D\cap\Delta(r)}g_{r}(o,x)\alpha^{m-1}. 
    \end{eqnarray*}
    Moreover, the simple counting  function of $f$ with respect to $D$ is defined by
     \begin{eqnarray*}
  \overline{N}_f(r, D)&=&\frac{\pi^m}{(m-1)!}\int_{f^{-1}(D)\cap\Delta(r)}g_{r}(o,x)\alpha^{m-1}. 
    \end{eqnarray*}

    \begin{lemma}[Green-Dynkin Formula, \cite{at2018a, Dong, DY}]\label{dynkin} Let $u$ be a $\mathscr C^2$-class function on  $M$ outside a polar set of singularities at most. Assume that $u(o)\not=\infty.$  Then
$$\int_{\partial \Delta(r)}u(x)d\pi_{r}(x)-u(o)=\frac{1}{2}\int_{\Delta(r)}g_r(o,x)\Delta u(x)dv(x).$$
\end{lemma}

We have the first main theorem as follows. 
\begin{theorem}[First Main Theorem]\label{first}  Assume that $f(o)\not\in{\rm Supp}D.$ Then  
$$T_f(r, \omega)+u_D\circ f(o)=m_f(r, D)+N_f(r, D).$$
\end{theorem}
\begin{proof} 
Let $\mathscr O(D)$ be the holomorphic line bundle over $N$ defined by $D.$
Let $s_D$ be the  canonical section  associated to $D,$ i.e., $s_D$ is a holomorphic section  of  $\mathscr O(D)$ over $N$ with zero divisor $D.$
Locally, we write  $s_D=\tilde s_D e,$ where $e$ is a local holomorphic frame of $\mathscr O(D).$ By Poincar\'e-Lelong formula, we obtain   
$[D]=dd^c[\log|\tilde s_D|^2]>0$
in the sense of currents. Furthermore,  it yields from  (\ref{888}) that 
$$\omega=dd^c\big[2u_D+\log|\tilde s_D|^2\big]>0,$$
which leads to  that $2u_D+\log|\tilde s_D|^2$ is  plurisubharmonic.
Since $N$ is a K\"ahler manifold, we deduce   that  both $\log|\tilde s_D|^2$ and $2u_D+\log|\tilde s_D|^2$ are  subharmonic. 
Thus,  $u_D$ is the difference of two subharmonic functions. It implies that the set of poles of $u_D\circ f$ is polar.
So, applying Green-Dynkin formula to $u_D\circ f,$  we obtain   
$$\int_{\partial\Delta(r)}u_D\circ fd\pi_{r}-u_D\circ f(o)=\frac{1}{2}\int_{\Delta(r)}g_{r}(o,x)\Delta (u_D\circ f)dv.$$
The first term on the left hand side equals $m_f(r, D)$ and the right hand side is  that 
 \begin{eqnarray*}
&& \frac{1}{2}\int_{\Delta(r)}g_{r}(o,x)\Delta (u_D\circ f)dv \\
 &=& \frac{2\pi^m}{(m-1)!}\int_{\Delta(r)}g_{r}(o,x)dd^c(u_D\circ f)\wedge\alpha^{m-1}  \\
  &=&\frac{\pi^m}{(m-1)!}\int_{\Delta(r)}g_{r}(o,x)f^*\omega\wedge\alpha^{m-1} \\
  && -\frac{\pi^m}{(m-1)!}\int_{\Delta(r)}g_{r}(o,x)dd^c\left[\log|\tilde s_D\circ f|^2\right]\wedge\alpha^{m-1} \\
 &=& T_f(r, \omega)-N_f(r, D).
    \end{eqnarray*}
    This completes the proof.
\end{proof}

\subsection{Second Main Theorem}~

We work on  two typical types of  K\"ahler manifolds  with  either non-positive sectional curvature or non-negative Ricci curvature. On such manifolds, we would like to  establish the second main theorem of meromorphic mappings. 

\noindent\textbf{{A. Non-positively Curved Case}}~

Let $(M, g)$ be a complete non-compact K\"ahler manifold 
 with non-positive sectional  curvature.    Recall that  $\kappa(r)$ stands for   the lower bound of  the Ricci curvature of $M$ on $B(r),$  defined by  (\ref{ricci}) in Introduction.   
In this part, we set $\Delta(r)=B(r)$ with $r>0.$

    Put 
  \begin{equation}\label{chi}
 \chi(s, t)=\left\{
                \begin{array}{ll}
t, \  \ & s=0; \\
  \frac{\sinh s t}{s}, \ \ & s\not=0.
              \end{array}
              \right.
               \end{equation}
 Consider  the following  Jacobi equation on $[0,\infty)$:
   \begin{equation}\label{G}
 G''+\kappa(t)G=0; \ \ \ \    G(0)=0, \ \ \ \  G'(0)=1.
   \end{equation}
 Note that this equation  has a unique  continuous solution. 

We give two-sided estimates of $G(t)$ as follows. 
\begin{lemma}\label{est} For $t\geq0,$ we have 
   $$ \chi(0, t)\leq G(t)\leq  \chi\left(\sqrt{-\kappa(t)}, t\right),$$
   where $\chi(s,t)$ is defined by $(\ref{chi}).$
\end{lemma}

\begin{proof}   Treat   the following  Jacobi equation on $[0,\infty)$:
 $$H_1''=0; \ \ \ \    H_1(0)=0, \ \ \ \  H_1'(0)=1,$$
which is uniquely solved by  $H_1(t)=t.$ Since  $\kappa(t)\leq0$ for $t\geq0, $   the standard comparison argument in ODE theory   shows that   
$$G(t)\geq H_1(t)=t=\chi(0,t).$$
Fix any  number $t_0>0.$ It is clear  that $\kappa(t)\geq\kappa(t_0)$ for $0\leq t\leq t_0.$
Treat   the following initial value  problem on $[0, t_0]$:
 $$H_2''+\kappa(t_0)H_2=0; \ \ \ \    H_2(0)=0, \ \ \ \  H_2'(0)=1,$$
which gives a unique solution $H_2(t)=\chi(\sqrt{-\kappa(t_0)}, t).$ Applying  the standard comparison argument (see \cite{P1}, Theorems 1, 2 for instance), 
we deduce that 
 $$G(t)\leq H_2(t)=\chi\left(\sqrt{-\kappa(t_0)},t\right)$$
for $0\leq t\leq t_0.$  In particular,  
  $G(t_0)\leq \chi(\sqrt{-\kappa(t_0)},t_0).$ 
Using the arbitrariness  of $t_0,$ we deduce that 
    $$G(t)\leq \chi\left(\sqrt{-\kappa(t)},t\right)$$ for $t\geq0.$ This completes the proof. 
\end{proof}

    \begin{lemma}[\cite{at2018a}]\label{zz} Let $\eta>0$ be a  constant. Then,  there exists  a constant $c_1>0$ such that
  $$g_r(o,x)\geq c_1\frac{\displaystyle\int_{\rho(x)}^rG(t)^{1-2m}dt}{\displaystyle\int_{\eta}^rG(t)^{1-2m}dt}$$
  holds for all $x\in B(r)\setminus \overline{B(2\eta)}$ with $r>3\eta.$ In particular, if $M$ is non-parabolic, then 
  there exists a constant $c_2>0$ such that 
  $$g_r(o,x)\geq c_2\int_{\rho(x)}^rG(t)^{1-2m}dt$$
  holds for all $x\in B(r),$ where $G(t)$ is defined by $(\ref{G}).$
  \end{lemma}
    
    For the upper estimate  of $g_r(o,x),$ we have the following well-known fact. 
    
 \begin{lemma}[\cite{Deb}]\label{sing}  We have 
$$g_r(o,x)\leq\left\{
                \begin{array}{ll}
                  \frac{1}{\pi}\log\frac{r}{\rho(x)}, \  \   & m=1; \\
                  \frac{1}{(m-1)\omega_{2m-1}}\big{(}\rho(x)^{2-2m}-r^{2-2m}\big{)},  \ \     & m\geq2  \\
                \end{array}
              \right. $$
              and
$$d\pi_r(x)\leq\frac{r^{1-2m}}{\omega_{2m-1}}d\sigma_r(x),  \  \ \ \ \ \ \ \ \ \ \ \ \  \ \   \ \ \ \ \  \ \ 
\  \ \ \ \ \ \ \ \ \ \ \ \  $$
 where  $\omega_{2m-1}$ is the  standard Euclidean    area  of  the unit sphere in $\mathbb R^{2m}.$ 
\end{lemma}

We  need the  following Borel's growth lemma (see \cite{No, ru}): 

 \begin{lemma}\label{} Let $u\geq0$ be a non-decreasing  function  on $(r_0, \infty)$ with $r_0\geq0.$ Then for any $\delta>0,$ there exists a subset $E_\delta\subset(r_0,\infty)$
 of finite Lebesgue measure such that  
 $$u'(r)\leq u(r)^{1+\delta}$$
 holds for all $r>r_0$ outside $E_{\delta}.$  
 \end{lemma}
 \begin{proof} The conclusion is clearly true  for $u\equiv0.$ Next, we assume that $u\not\equiv0.$
 Since $u\geq0$ is a non-decreasing  function,  there exists a number $r_1>r_0$ such that $u(r_1)>0.$  The non-decreasing property of $u$ implies that  the  limit
$\eta:=\lim_{r\to\infty}u(r)$
exists or $\eta=\infty.$  If $\eta=\infty,$ then $\eta^{-1}=0.$ 
 Set  
 $$E_\delta=\left\{r\in(r_0,\infty):  \ u'(r)>u(r)^{1+\delta}\right\}.$$
Since $u$ is a non-decreasing function on $(r_0, \infty),$  we deduce that  $u'(r)$ exists for  almost all  $r\in(r_0, \infty).$   It is therefore  
   \begin{eqnarray*}
\int_{E_\delta}dr 
 &\leq& \int_{r_0}^{r_1}dr+\int_{r_1}^\infty\frac{u'(r)}{u(r)^{1+\delta}}dr \\
 &=&\frac{1}{\delta u(r_1)^\delta}-\frac{1}{\delta \eta^\delta}+r_1-r_0\\
 &<&\infty.
    \end{eqnarray*}
 This completes the proof. 
 \end{proof}
 
  Set 
  \begin{equation}\label{kkk}
  K(r,\delta)=\frac{r^{1-2m}\bigg(\displaystyle\int_{\frac{1}{3}}^rG(t)^{1-2m}dt\bigg)^{(1+\delta)^2}}{G(r)^{(1-2m)(1+\delta)}},
    \end{equation}
where $G(t)$ is defined by (\ref{G}).

With the previous preparations,  we  establish a calculus lemma: 
\begin{theorem}\label{cal} Let $k\geq0$ be a locally integrable function on $M.$ Assume that $k$ is locally bounded at $o.$ Then 
 for any   $\delta>0,$  there exists   a subset $E_{\delta}\subset(1,\infty)$ of finite Lebesgue measure such that
$$\int_{\partial B(r)}kd\pi_r\leq C K(\delta, r)
      \bigg(\int_{B(r)}g_r(o,x)kdv\bigg)^{(1+\delta)^2}$$
holds for all  $r>1$ outside $E_{\delta},$ where    $K(r,\delta)$ is defined by $(\ref{kkk})$ and $C>0$ is a sufficiently large  constant independent of  $k, \delta, r.$ 
 \end{theorem}
\begin{proof} 
 For $r>1,$  it yields from Lemma \ref{sing} that 
      \begin{eqnarray*}
\Lambda(r) &:=& \int_{B(r)}g_r(o,x)kdv  \\
&\geq& \int_1^rdt\int_{\partial B(t)}g_r(o,x)kd\sigma_t \\
      &\geq& \omega_{2m-1} \int_1^rt^{2m-1}dt\int_{\partial B(t)}g_r(o,x)kd\pi_t.  
      \end{eqnarray*}
       According to Lemma \ref{zz},  there exists  a constant $c>0$ such that
  $$g_r(o,x)\geq c\frac{\displaystyle\int_{\rho(x)}^rG(t)^{1-2m}dt}{\displaystyle\int_{1/3}^rG(t)^{1-2m}dt}$$
  holds all $x\in B(r)\setminus \overline{B(2/3)}$ with $r>1.$
Whence, for $r>t>1$  
$$g_r(o,x)\big|_{\partial B(t)}\geq c\frac{\displaystyle\int_{t}^rG(s)^{1-2m}ds}{\displaystyle\int_{\frac{1}{3}}^rG(s)^{1-2m}ds}.$$  
       It is therefore 
                 \begin{eqnarray*}
   & &   \int_{\partial B(t)}g_r(o,x)kd\pi_t  \\
      &\geq& c\bigg(\displaystyle\int_{\frac{1}{3}}^rG(s)^{1-2m}ds\bigg)^{-1} \int_{\partial B(t)}kd\pi_t\displaystyle\int_{t}^rG(s)^{1-2m}ds. 
                       \end{eqnarray*}
Combining the above to get  
      \begin{eqnarray*}
&& \Gamma(r) \\ 
      &\geq& c\omega_{2m-1}\bigg(\displaystyle\int_{\frac{1}{3}}^rG(s)^{1-2m}ds\bigg)^{-1} \int_1^rt^{2m-1}dt\int_{\partial B(t)}kd\pi_t\displaystyle\int_{t}^rG(s)^{1-2m}ds \\
            &:=& \Lambda(r), 
      \end{eqnarray*}
    which  leads to  
                           \begin{eqnarray*}
\frac{\displaystyle\frac{d}{dr}\bigg(\Lambda(r)\displaystyle\int_{\frac{1}{3}}^rG(t)^{1-2m}dt\bigg)}{G(r)^{1-2m}} &=&  c\omega_{2m-1}\int_1^rt^{2m-1}dt\int_{\partial B(t)}kd\pi_t.
                       \end{eqnarray*}
In further, we obtain 
       $$ \frac{d}{dr}\frac{\displaystyle\frac{d}{dr}\bigg(\Lambda(r)\displaystyle\int_{\frac{1}{3}}^rG(t)^{1-2m}dt\bigg)}{G(r)^{1-2m}}=c\omega_{2m-1}r^{2m-1}\int_{\partial B(r)}kd\pi_r.$$                
       Using  Borel's growth lemma twice,  then for any $\delta>0,$  there exists a subset $E_\delta\subset(1,\infty)$ of finite Lebesgue measure such that   
            \begin{eqnarray*}
     \int_{\partial B(r)}kd\pi_r  
    &\leq& C\frac{\displaystyle r^{1-2m}\bigg(\displaystyle\int_{\frac{1}{3}}^rG(t)^{1-2m}dt\bigg)^{(1+\delta)^2}}{G(r)^{(1-2m)(1+\delta)}}\Lambda(r)^{(1+\delta)^2} \\
     &\leq& CK(r,\delta)\Gamma(r)^{(1+\delta)^2}
               \end{eqnarray*}
  holds for  all $r>1$ outside $E_\delta,$ where $C=1/c\omega_{2m-1}>0$ is  clearly a constant independent of  $k, \delta, r.$ 
  This completes the proof. 
\end{proof}

By  estimating $\log^+K(r,\delta),$ we further obtain: 

\begin{cor}\label{cal1} Let $k\geq0$ be a locally integrable function on $M.$ Assume that $k$ is locally bounded at $o.$ Then 
 for any   $\delta>0,$  there exists   a subset $E_{\delta}\subset(1,\infty)$ of finite Lebesgue measure such that
 \begin{eqnarray*}
&& \log^+\int_{\partial B(r)}kd\pi_r  \\
&\leq& 
      (1+\delta)^2\log^+\int_{B(r)}g_r(o,x)kdv+O\left(\sqrt{-\kappa(r)}r+\delta\log r\right)
      \end{eqnarray*} 
holds for all  $r>1$ outside $E_{\delta},$  where $\kappa$ is defined by $(\ref{ricci}).$     
 \end{cor}
\begin{proof}  
 By  Lemma \ref{est}, we have  $G(t)\geq t$ for $t>0.$ Thus, one has
   \begin{eqnarray*}
\int_{\frac{1}{3}}^rG(t)^{1-2m}dt &\leq& \int_{1}^r t^{1-2m}dt+O(1) \\
&\leq& \log r+O(1)
\end{eqnarray*}
for $r>1.$
It yields   that 
\begin{equation}\label{x1}
\log^+\left(\int_{\frac{1}{3}}^rG(t)^{1-2m}dt\right)^{(1+\delta)^2}\leq (1+\delta)^2\log^+\log r+O(1).
\end{equation}
If $\kappa\equiv0,$ then $G(r)\leq r$ due to Lemma \ref{est}. Thus, we obtain  
$$r^{1-2m}G(r)^{(2m-1)(1+\delta)}\leq r^{(2m-1)\delta},$$
which leads to  
\begin{equation}\label{x2}
\log^+\left(r^{1-2m}G(r)^{(2m-1)(1+\delta)}\right)\leq (2m-1)\delta \log r
\end{equation}
for $r>1.$ Combining (\ref{x1}) with (\ref{x2}), we conclude that 
$$\log^+K(r,\delta)\leq O(\delta\log r).$$
By Theorem \ref{cal}, we have  the conclusion holds. 
If $\kappa\not\equiv0,$ then  Lemma \ref{est}  gives that 
\begin{eqnarray}\label{x3}
&& \log^+\left(r^{1-2m}G(r)^{(2m-1)(1+\delta)}\right) \\
&\leq&\log^+G(r)^{(2m-1)(1+\delta)} \nonumber  \\
&\leq& O\left(\sqrt{-\kappa(r)}r+1\right). \nonumber
\end{eqnarray}
Combining (\ref{x1}) with (\ref{x3}), we conclude that 
$$\log^+K(r,\delta)\leq O\left(\sqrt{-\kappa(r)}r+1\right).$$
Hence,  the conclusion also  holds  due to   Theorem \ref{cal}.  
\end{proof}

Let $D_1,\cdots, D_q$ be  effective divisors  on $N$ such that each $D_j$ is cohomologous to $\omega,$ i.e., there are  functions $u_{D_1},\cdots, u_{D_q}\geq0$ on $N$ such that 
$$\omega-[D_j]=2dd^cu_{D_j}, \ \ \ \  j=1,\cdots,q.$$ 
Set
$$\Psi=\frac{\omega^n}{\prod_{j=1}^q u_{D_j}^{-2}e^{-2u_{D_j}}}.$$

\begin{lemma}\cite{gri}\label{grii}
Assume that   $D_1, \cdots, D_q$ are  in general position. 
If $qw-{\rm Ric}\Psi>0,$ then 

$(a)$  $-{\rm Ric}\Psi\geq0;$

$(b)$ there exists a constant $c>0$ such that $(-{\rm Ric}\Psi)^n\geq c\Psi;$

$(c)$ $\int_{N\setminus D}(-{\rm Ric}\Psi)^n<\infty,$ where $D=D_1+\cdots+D_q.$
\end{lemma}

  \begin{theorem}\label{main1}  Let $M$ be a  
     complete non-compact  K\"ahler manifold   with non-positive sectional curvature.  
Let $(N,\omega)$ be a  compact  K\"ahler  manifold of complex dimension not greater than  that  of $M.$
 Let $D_1,\cdots,D_q$ be effective divisors in general position on $N$ such that each $D_j$ is cohomologous to  $\omega.$
  Let  $f:M\rightarrow N$ be a  differentiably non-degenerate meromorphic mapping.  
  Assume that $q\omega-{\rm Ric}(\omega^n)>0.$
  Then  for any  $\delta>0,$ there exists a subset $E_\delta\subset(1, \infty)$ of finite Lebesgue measure such that 
        \begin{eqnarray*}
&&qT_f(r,\omega)+T_f(r, K_N)+T(r, \mathscr R) \\
&\leq& \sum_{j=1}^q\overline N_f(r,D_j)+O\big(\log T_{f}(r, \omega)+\sqrt{-\kappa(r)}r+\delta\log r\big)
         \end{eqnarray*}
    holds for all  $r>1$ outside $E_\delta,$  where $\kappa$ is defined by $(\ref{ricci}).$          
\end{theorem}  

\begin{proof}
Define a non-negative function $\xi$ by 
$$f^*\Psi\wedge\alpha^{m-n}=\xi\alpha^m.$$
Using  Poincar\'e-Lelong formula, we obtain   
 \begin{eqnarray*}
&& dd^c\left[\log\xi\right] \\
&=& -f^*{\rm Ric}(\omega^n)+\mathscr R+[D_{f,\rm ram}]+2\sum_{j=1}^q dd^c(u_{D_j}\circ f)+2\sum_{j=1}^q dd^c\log(u_{D_j}\circ f) \\
&=& -f^*{\rm Ric}(\omega^n)+\mathscr R+[D_{f,\rm ram}]+\sum_{j=1}^q\left(f^*\omega-f^*[D_j]\right)
+2\sum_{j=1}^q dd^c\log(u_{D_j}\circ f) \\
&=& qf^*\omega-f^*{\rm Ric}(\omega^n)+\mathscr R-f^*[D]+[D_{f,\rm ram}]+2\sum_{j=1}^q dd^c\log(u_{D_j}\circ f).
 \end{eqnarray*}
Since $-[D_j]\leq 2dd^c u_{D_j}\leq\omega$ with  $j=1,\cdots,q,$
 it  follows from Green-Dynkin formula and the first main theorem that  
 \begin{eqnarray}\label{gjj}
&& \frac{1}{2}\int_{\partial B(r)}\log \xi d\pi_r-\frac{1}{2}\log \xi(o) \\
&=&\frac{1}{4}\int_{B(r)}\Delta\log\xi dv \nonumber \\
 &\geq& qT_f(r,\omega)+T_f(r, K_N)+T(r,\mathscr R)-N_f(r, D) +N(r, D_{f, \rm ram}) \nonumber \\
 && +~\frac{1}{2}\sum_{j=1}^q\int_{B(r)}g_r(o,x)\Delta \log(u_{D_j}\circ f) dv+O(1)  \nonumber \\
&\geq& qT_f(r,\omega)+T_f(r, K_N)+T(r,\mathscr R)-\overline{N}_f(r, D)+O\big(\log T_f(r,\omega)\big). \nonumber
 \end{eqnarray}
Write
$$-f^*{\rm Ric}\Psi=\frac{\sqrt{-1}}{\pi}\sum_{i,j=1}^m\psi_{i\bar j}dz_i\wedge d\bar z_j.$$
For any $x_0\in M,$ one can take a   normal local holomorphic coordinate $z$  near $x_0$ such that  
$$\alpha|_{x_0}=\frac{\sqrt{-1}}{\pi}\sum_{j=1}^m dz_j\wedge d\bar z_j.$$
By Lemma \ref{grii}
 \begin{eqnarray*}
f^*\Psi\wedge\alpha^{m-n}\big|_{x_0}&\leq& c^{-1}\left(-f^*{\rm Ric\Psi}\right)^n\wedge\alpha^{m-n} \\
&=& \bigg(\frac{\sqrt{-1}}{\pi}\sum_{i,j=1}^m\psi_{i\bar j}dz_i\wedge d\bar z_j\bigg)^{n}\bigwedge \bigg( \frac{\sqrt{-1}}{\pi}\sum_{j=1}^m dz_j\wedge d\bar z_j\bigg)^{m-n} \\
&=& c^{-1}(m-n)!\sum_{1\leq i_1\not=\cdots\not=i_n\leq m}\psi_{i_1\bar i_1}\cdots\psi_{i_n\bar i_n}\alpha^m \\
&\leq& c^{-1}(m-n)! \left({\rm tr}(\psi_{i\bar j})\right)^n\alpha^m.
 \end{eqnarray*}
On the other hand, we have 
\begin{equation}
-f^*{\rm Ric\Psi}\wedge\alpha^{m-1}\big|_{x_0}=(m-1)!{\rm tr}(\psi_{i\bar j})\alpha^{m},
\end{equation}
which  leads to 
 \begin{eqnarray*}
\xi^{\frac{1}{n}}\big|_{x_0}&=&\left(\frac{f^*\Psi\wedge\alpha^{m-n}}{\alpha^m}\right)^{\frac{1}{n}} \\
&\leq& c^{-\frac{1}{n}}(m-n)!^{\frac{1}{n}} {\rm tr}(\psi_{i\bar j}) \\
&=&-\frac{c^{-\frac{1}{n}}(m-n)!^{\frac{1}{n}}}{(m-1)!}\frac{f^*{\rm Ric\Psi}\wedge\alpha^{m-1}}{\alpha^m}.
 \end{eqnarray*}
By the arbitrariness  of $x_0,$ we show that 
$$\xi^{\frac{1}{n}}\leq -\frac{c^{-\frac{1}{n}}(m-n)!^{\frac{1}{n}}}{(m-1)!}\frac{f^*{\rm Ric\Psi}\wedge\alpha^{m-1}}{\alpha^m}.$$
 Using Corollary \ref{cal1},  it is therefore    
  \begin{eqnarray*}
 \int_{\partial B(r)}\log \xi d\pi_r \
&=&  n\int_{\partial B(r)}\log \xi^{\frac{1}{n}} d\pi_r \\
&\leq&  n\log\int_{\partial B(r)}\xi^{\frac{1}{n}} d\pi_r\ \\
&\leq& n(1+\delta)^2\log \int_{B(r)}g_r(o,x)\xi^{\frac{1}{n}}dv+O\left(\sqrt{-\kappa(r)}r+\delta\log r\right) \\
&\leq& n(1+\delta)^2\log \int_{B(r)}g_r(o,x)\frac{-f^*{\rm Ric\Psi}\wedge\alpha^{m-1}}{\alpha^m}dv \\
&& +~O\left(\sqrt{-\kappa(r)}r+\delta\log r\right) \\
&=& n(1+\delta)^2\log \int_{B(r)}g_r(o,x) (-f^*{\rm Ric\Psi})\wedge\alpha^{m-1}\\ 
&&+~O\left(\sqrt{-\kappa(r)}r+\delta\log r\right) \\
&=& n(1+\delta)^2\log T_f(r, -{\rm Ric}\Psi)+O\left(\sqrt{-\kappa(r)}r+\delta\log r\right).
 \end{eqnarray*}
Since $N$ is compact and  $\omega>0,$ we get 
$$T_f(r, -{\rm Ric}\Psi)\leq O\big(T_f(r, \omega)+1\big).$$
This implies that 
 $$ \int_{\partial B(r)}\log \xi d\pi_r \leq O\left(\log T_f(r, \omega)+\sqrt{-\kappa(r)}r+\delta\log r\right). $$
By this with (\ref{gjj}), we prove the theorem.
\end{proof}

\noindent\textbf{{B. Ricci Non-negatively Curved  Case}}~

 Assume   that  $M$ is  non-parabolic, with non-negative Ricci curvature. 
  Let    $V(r)$ denote  the Riemannian volume of geodesic ball $B(r)$   centered at $o$ with radius  $r$ in $M.$
       The non-parabolicity of $M$ means  that
 $$\int_1^\infty\frac{t}{V(t)}dt<\infty.$$
Thus, we have the unique minimal positive global Green function $G(o,x)$ of $\Delta/2$  for $M,$  
which can be written as 
$$G(o, x)=2\int_0^\infty p(t, o,x)dt,$$
where $p(t,o,x)$ is the heat kernel of $\Delta$ on $M.$
  Let $\rho(x)$ be the Riemannian distance function of $x$ from $o.$  By  Li-Yau's estimate \cite{Li-Yau}, there exist constants $A, B>0$ such that  
 \begin{equation}\label{Gr}
 A\int_{\rho(x)}^\infty\frac{t}{V(t)}dt\leq G(o,x)\leq B\int_{\rho(x)}^\infty\frac{t}{V(t)}dt
 \end{equation}
holds for all $x\in M.$ 
\subsubsection{Construction of $\Delta(r)$}~\label{sec421}

 For $r>0,$ define  
 \begin{equation}\label{delta}
 \Delta(r)=\left\{x\in M: \    G(o,x)>A\int_r^\infty\frac{t}{V(t)}dt\right\}.
  \end{equation} 
Since    
$\lim_{x\to o}G(o,x)=\infty$ and   $\lim_{x\to\infty}G(o,x)=0,$  we conclude   that 
    $\Delta(r)$ is a precompact  domain containing $o$  satisfying   that 
 $ \lim_{r\to0}\Delta(r)\to \emptyset$ and  $\lim_{r\to\infty}\Delta(r)=M.$
Moreover,  the family 
     $\{\Delta(r)\}_{r>0}$ exhausts $M.$    
Hence,  the boundary $\partial\Delta(r)$ of $\Delta(r)$ can be formulated as
 $$\partial\Delta(r)=\left\{x\in M: \    G(o,x)=A\int_r^\infty\frac{t}{V(t)}dt\right\}.$$
  By   Sard's theorem,   $\partial\Delta(r)$  is a submanifold of $M$ for almost all $r>0.$  
    
     Set
 $$g_r(o,x)=G(o,x)-A\int_r^\infty\frac{t}{V(t)}dt.$$
Note that   $g_r(o,x)$ defines   the  Green function of $\Delta/2$ for $\Delta(r)$ with a pole at $o$ satisfying Dirichelet boundary condition, i.e., 
  $$-\frac{1}{2}\Delta g_r(o,x)=\delta_o(x), \ \ \ \  ^\forall x\in\Delta(r);  \ \ \  \   \  g_r(o,x)=0, \ \ \  \ ^\forall x\in\partial\Delta(r).$$
 Let  $\pi_r$ denote the harmonic measure  on $\partial\Delta(r)$ with respect to $o,$ defined by
  $$d\pi_r(x)=\frac{1}{2}\frac{\partial g_r(o,x)}{\partial{\vec{\nu}}}d\sigma_r(x), \ \ \ \   ^\forall x\in\partial\Delta(r),$$
  where  $\partial/\partial \vec\nu$ is the inward  normal derivative on $\partial \Delta(r),$ $d\sigma_{r}$ is the Riemannian area element of 
$\partial \Delta(r).$

   \begin{lemma}\label{thm4} For $x\in\partial\Delta(t),$ we have
 $$g_r(o,x)=A\int_t^r \frac{s}{V(s)}ds$$
holds  for all $r\geq t,$ where $A$ is given by $(\ref{Gr}).$
  \end{lemma}
 \begin{proof}    According to the definition of Green function for $\Delta(r),$  it is immediate that  for $r\geq t$
  \begin{eqnarray*} 
g_r(o,x)&=& G(o,x)-A\int_{r}^\infty\frac{t}{V(t)}dt   \\
&=& G(o,x)-A\int_{t}^\infty\frac{s}{V(s)}ds +A\int_{t}^r\frac{s}{V(s)}ds  \\
 &=& g_t(o,x)+A\int_{t}^r\frac{s}{V(s)}ds.   
   \end{eqnarray*}
 Since
 $g_t(o,x)=0$ for  $x\in\partial\Delta(t),$
  we have the lemma proved.
 \end{proof}

 Let $\nabla$ denote the gradient operator on any Riemannian manifold.  Cheng-Yau \cite{C-Y} proved the following lemma.

 \begin{lemma}\label{CY}
 Let $N$ be a complete Riemannian manifold.  Let   $B(x_0, r)$ be a geodesic ball centered at $x_0\in N$ with radius $r.$ 
 Then, there exists  a constant $c_N>0$ depending only on the dimension of $N$ such that 
 $$\frac{\|\nabla u(x)\|}{u(x)}\leq \frac{c_N r^2}{r^2-d(x_0, x)^2}\Big(|\kappa(r)|+\frac{1}{d(x_0, x)}\Big)$$
 holds for any  non-negative  harmonic function $u$ on $B(x_0, r),$ 
 where $d(x_0, x)$ is the Riemannian distance between $x_0$ and $x,$ and $\kappa(r)$ is the lower bound of  Ricci curvature of  $B(x_0, r).$
 \end{lemma}

 \begin{theorem}\label{hh} There exists a constant  $c_1>0$   such that 
 $$\|\nabla g_r(o,x)\|\leq \frac{c_1}{r}\int_{r}^\infty\frac{tdt}{V(t)}, \ \ \ \    ^\forall x\in\partial\Delta(r).$$
\end{theorem}
\begin{proof} By the curvature assumption,   ${\rm{Ric}}_M\geq0.$  It follows  from Lemma \ref{CY} and (\ref{Gr})  that   (see Remark 5 in \cite{Sa0} also)
  \begin{eqnarray*}
\|\nabla G(o, x)\|\leq\frac{c_0}{\rho(x)}G(o,x) 
\leq \frac{c_0B}{\rho(x)}\int_{\rho(x)}^\infty\frac{tdt}{V(t)}
  \end{eqnarray*}
 for some large  constant $c_0>0$ which depends  only on  the dimension  $m.$  By  (\ref{Gr})   again 
 $$\int_{\rho(x)}^\infty\frac{tdt}{V(t)}\leq \int_r^\infty \frac{tdt}{V(t)}, \ \ \  \ ^\forall x\in \partial\Delta(r),$$
 which  gives    
 $$\rho(x)\geq r, \ \ \ \     ^\forall x\in\partial\Delta(r).$$
Set $c_1=c_0B,$ then we conclude that 
$$\|\nabla g_r(o, x)\|\big|_{\partial\Delta(r)}=\|\nabla G(o, x)\|\big|_{\partial\Delta(r)}\leq \frac{c_1}{r}\int_{r}^\infty\frac{tdt}{V(t)}.$$
  \end{proof}
   
   By
   $$\frac{\partial g_r(o,x)}{\partial{\vec{\nu}}}=\|\nabla g_r(o, x)\|,$$
we  obtain   an estimate for  the upper bound of  $d\pi_r$  as follows. 
  \begin{cor}\label{bbbq}  There exists a constant  $c_2>0$   such that 
   $$d\pi_r\leq \frac{c_2}{r}\int_{r}^\infty\frac{tdt}{V(t)}d\sigma_r,$$
 where  $d\sigma_{r}$ is the Riemannian area element of 
$\partial \Delta(r).$
\end{cor}

 Set
 \begin{equation}\label{Hr1}
H(r,\delta)=\frac{1}{r}\left(\frac{V(r)}{r}\right)^{1+\delta}\int_{r}^\infty\frac{tdt}{V(t)}. 
 \end{equation} 

 \begin{theorem}\label{calculus}
Let $k\geq0$ be a locally integrable function on $M.$ Assume that $k$ is locally bounded at $o.$ Then there exists a  constant $C>0$ such that
 for any $\delta>0,$ there exists  a subset $E_{\delta}\subset(0,\infty)$ of finite Lebesgue measure such that
$$\int_{\partial\Delta(r)}kd\pi_r\leq CH(r,\delta)\bigg(\int_{\Delta(r)}g_r(o,x)kdv\bigg)^{(1+\delta)^2}$$
holds for all $r>0$ outside $E_{\delta},$ where $H(r,\delta)$ is given by $(\ref{Hr1}).$  
\end{theorem}
 
  \begin{proof} 
 Invoking  Lemma  \ref{thm4}, we have   
   \begin{eqnarray*}
 \int_{\Delta(r)}g_r(o,x)kdv 
 &=&\int_0^rdt \int_{\partial\Delta(t)}g_r(o,x)kd\sigma_t \\
    &=&  A\int_0^r\left(\int_{t}^r\frac{s}{V(s)}ds\right)dt \int_{\partial\Delta(t)}kd\sigma_t.
   \end{eqnarray*}
 Set 
 $$\Lambda(r)=A\int_0^{r}\left(\int_{t}^{r}\frac{s}{V(s)}ds\right)dt \int_{\partial\Delta(t)}kd\sigma_t.$$
 A simple computation leads to  
 $$\Lambda'(r)=\frac{d\Lambda(r)}{dr}=\frac{Ar}{V(r)}\int_0^rdt\int_{\partial\Delta(t)}kd\sigma_t.$$
 In further, we have
 $$\frac{d}{dr}\left(\frac{V(r)\Lambda'(r)}{r}\right)=A\int_{\partial\Delta(r)}kd\sigma_r.$$
Apply  Borel's lemma  to the left hand side of the above equality  twice:  one is to $V(r)\Lambda'(r)/r$ and  the other  is to 
 $\Lambda'(r),$ then  we deduce  that for any  $\delta>0,$   there exists a subset $E_\delta\subset(0,\infty)$ of finite Lebesgue measure such that 
$$ \int_{\partial\Delta(r)}kd\sigma_r\leq \frac{1}{A}\left(\frac{V(r)}{r}\right)^{1+\delta}\Lambda(r)^{(1+\delta)^2}
$$ holds for all $r>0$ outside $E_\delta.$  
  Moreover,  Corollary \ref{bbbq} implies that there exists a constant $c>0$ such that
  $$d\pi_r\leq \frac{c}{r}\int_{r}^\infty\frac{tdt}{V(t)}d\sigma_r.$$
Set $C=c/A.$ Combining the above, we have 
   \begin{eqnarray*}
 \int_{\partial\Delta(r)}kd\pi_r&\leq& \frac{c}{Ar}\left(\frac{V(r)}{r}\right)^{1+\delta}\int_{r}^\infty\frac{tdt}{V(t)}\Lambda(r)^{(1+\delta)^2} \\
 &=&CH(r,\delta)\Lambda(r)^{(1+\delta)^2}.
   \end{eqnarray*}
 holds for all $r>0$ outside $E_\delta.$  
 \end{proof}

  \begin{theorem}\label{main2}  Let $M$ be a  non-parabolic 
     complete non-compact  K\"ahler manifold   with non-negative  Ricci curvature.  
Let $(N,\omega)$ be a  compact  K\"ahler  manifold of complex dimension not greater than  that  of $M.$
 Let $D_1,\cdots,D_q$ be effective  divisors in general position on $N$ with that each  $D_j$ is cohomologous to  $\omega.$
  Let  $f:M\rightarrow N$ be a  differentiably non-degenerate meromorphic mapping. 
    Assume that $q\omega-{\rm Ric}(\omega^n)>0.$
   Then  for any  $\delta>0,$ there exists a subset $E_\delta\subset(1, \infty)$ of finite Lebesgue measure such that 
        \begin{eqnarray*}
&&qT_f(r,\omega)+T_f(r, K_N)+T(r, \mathscr R) \\
&\leq& \sum_{j=1}^q\overline N_f(r,D_j)+O\big(\log T_{f}(r, \omega)+\log H(r)+\delta\log r\big)
         \end{eqnarray*}
    holds for all  $r>1$ outside $E_\delta,$  where $H(r)$ is defined by $(\ref{Hr}).$          
\end{theorem}  
\begin{proof}
Applying  the similar  arguments as in the proof of Theorem \ref{main1} and using Theorem \ref{calculus},   the theorem can be proved.
\end{proof}

\section{Equidistribution of Algebroid  Mappings}

\subsection{First Main Theorem}~\label{61}

Let $(M, \alpha)$ be an $m$-dimensional complete non-compact K\"ahler manifold  with 
 Laplace-Beltrami operator $\Delta,$ and $(N, \omega)$ a  K\"ahler projective manifold of  $\dim_{\mathbb C}N=n\leq m.$
Let $F: M\to N$ be a $\nu$-valued algebroid mapping with $\nu$ distinct single-valued  components $F_1,\cdots, F_\nu.$ Set
$$F^*\omega=\sum_{j=1}^\nu F_j^*\omega, \ \ \ \   F^*u_D=\sum_{j=1}^\nu F_j^*u_D.$$
\ \ \ \  Given  a family $\{\Delta(r)\}_{r>0}$ of   precompact domains $\Delta(r)\subset M$ with   smooth boundaries $\partial\Delta(r)$ which  can exhaust  $M.$
       Fix a reference  point $o\in \Delta(r).$ We use  $g_{r}(o, x)$   to stand for  
              the  Green function of $\Delta/2$ for $\Delta(r)$ with a pole at $o$ satisfying Dirichlet boundary condition, and then  denote  by  $\pi_r$ 
               the harmonic measure on $\partial\Delta(r)$ with respect to $o.$

          We  shall define  Nevanlinna's functions of $F.$       Let  $D$ be  an effective  divisor on $N$  cohomologous to  $\omega,$  i.e., there exists a function $u_D\geq0$ on $N$ such that 
                             $\omega-[D]=2dd^cu_D$ in the sense of currents. 
              The characteristic function of $F$ with respect to $\omega$ is defined by
$$T_F(r,\omega)=\frac{\pi^m}{(m-1)!\nu}\int_{\Delta(r)}g_r(o,x)F^*\omega\wedge\alpha^{m-1}.$$
    The proximity  function of $F$ with respect to $D$ is defined  by
$$ m_F(r, D)=\frac{1}{\nu}\int_{\partial\Delta(r)} F^*u_Dd\pi_{r}.$$
The counting function   of  $F$ with respect to $D$ is defined by
$$  N_F(r, D)=\frac{\pi^m}{(m-1)!\nu}\int_{F^*D\cap\Delta(r)}g_{r}(o,x)\alpha^{m-1}. 
$$

\begin{theorem}[First Main Theorem]\label{}  Assume that $F(o)\not\in{\rm Supp}D.$ Then  
$$T_F(r, \omega)+O(1)=m_F(r, D)+N_F(r, D).$$
\end{theorem}

\begin{proof} 
Since  the set of singularities  of $F^*u_D$ is polar, 
  Green-Dynkin formula can  apply  to each  $F_j^*u_D,$ and thus to $F^*u_D.$  Using the similar arguments as in the proof of Theorem \ref{first},    we can prove the theorem.  
\end{proof}

\subsection{Counting Function of Branch Divisors}~

The counting function of  branch divisor $\mathscr D_F$ of $F$  is defined by
$$N_{{\rm{bran}}}(r, F)=\frac{1}{\nu}N(r,\mathscr D_F),$$ where 
$$N(r,\mathscr D_F)=\frac{\pi^m}{(m-1)!}\int_{\mathscr D_F\cap\Delta(r)}g_{r}(o,x)\alpha^{m-1}.$$
Recall that $J_F$ is defined by (\ref{JJ}).  Set 
$$N\big(r, (J_F=0)\big) = \frac{\pi^m}{(m-1)!}\int_{(J_F=0)\cap\Delta(r)}g_r(o,x)\alpha^{m-1}.$$
 
 Recall that   
$\jmath: N\hookrightarrow\mathbb P^d(\mathbb C)$ is  the  holomorphic   embedding, and $\omega_{FS}|_N$ is the restriction of $\omega_{FS}$ to $N.$
 
  \begin{theorem}[Estimate of branch divisors]\label{bran} We have 
$$N_{\rm{bran}}(r, F)\leq (2\nu-2)\left[\frac{\omega_{FS}|_N}{\omega}\right]T_F(r, \omega)+O(1).$$
 \end{theorem}
\begin{proof}  From (\ref{JJ}), we have  
   \begin{eqnarray*} 
   N\big(r, (J_\psi=0)\big)  
 &=&  \frac{1}{4}\int_{\Delta(r)}g_r(o,x)\Delta\log|J_F|^2dv \\
   &=&  \frac{\nu-1}{2}\int_{\Delta(r)}g_r(o,x)\Delta\log|f_{10}\cdots f_{\nu0}|^2dv \\
   &&+ \frac{1}{4}\int_{\Delta(r)}g_r(o,x)\Delta\log\prod_{1\leq i<j\leq\nu}\big|g_{ik}-g_{jk}\big|^4dv \\
   &=& \frac{\nu-1}{2}\sum_{i=1}^\nu\int_{\Delta(r)}g_r(o,x)\Delta\log|f_{i0}|^2dv \\ 
   &&+ \sum_{1\leq i<j\leq\nu}\int_{\partial\Delta(r)}\log|g_{ik}-g_{jk}|^2d\pi_{r}+O(1).
      \end{eqnarray*}
Using Green-Dynkin formula, we further obtain
   \begin{eqnarray*}
     N\big(r, (J_\psi=0)\big)      &\leq& (2\nu-2)\sum_{i=1}^\nu N(r, g_{ik})+(2\nu-2)\sum_{i=1}^\nu m(r, g_{ik})+O(1) \\
   &=&  (2\nu-2)\sum_{i=1}^\nu T(r, g_{ik})+O(1) \\
     &\leq& (2\nu-2)\sum_{i=1}^\nu T_{\jmath\circ F_i}(r, \omega_{FS})+O(1) \\
   &\leq& \nu(2\nu-2)\left[\frac{\omega_{FS}|_N}{\omega}\right] T_F(r, \omega)+O(1).
   \end{eqnarray*}
By this, we conclude  from Theorem \ref{esti} that 
   \begin{eqnarray*} 
N_{\rm{bran}}(r, F)&\leq& \frac{1}{\nu}N\big(r,(J_\psi=0)\big) \\
&\leq& (2\nu-2)\left[\frac{\omega_{FS}|_N}{\omega}\right]T_F(r, \omega)+O(1).
   \end{eqnarray*} 
\end{proof}

\subsection{Second Main Theorem}~\label{63}

 Let $\mathcal M_F$ be the $\nu$-leaf complex manifold by  a $\nu$-valued algebroid mapping  $F: M\to N.$ 
Let $\pi: \mathcal M_F\to M$ be the  natural  projection.  Equip  $\mathcal M_F$ with the  pullback metric $\pi^*g$ induced by $g$ on $M.$
 Note that  $\pi^*g$ is a positive  semi-definite K\"ahler metric  on $\mathcal M_F,$ i.e.,  
 $\pi^*g$  is a  K\"ahler  metric    on  $\mathcal M_F\setminus\pi^{-1}(\mathscr B_F)$ and  degenerate on $\pi^{-1}(\mathscr B_F).$
      Since   $\mathscr B_F$ is   a polar  set of ${\rm codim}_{\mathbb C}\mathscr B_F=1,$  the  Green-Dynkin formula still  works on each leaf of $(\mathcal M_F, \pi^*g)$  for  the following reasons. 
 Set $K=\pi^{-1}(\mathscr B_F).$   Take    a sequence $\{U_n(K)\}_{n=1}^\infty$ of  neighborhoods $U_n(K)$ of $K$ such that 
$$\cdots\subseteq \overline{U_{k+1}(K)}\subseteq U_k(K)\subseteq\cdots\subseteq \overline{U_2(K)}\subseteq U_1(K), \ \ \ \   \bigcap_{n\geq1}U_n(K)=K.$$
\ \ \ \  Let $\{\tau_{n}, \theta_{n}\}$ be a partition  of the unity   subordinate to the open covering $\{\mathcal M_F\setminus U_{n+1}(K), U_n(K)\}$ of $\mathcal M_F.$
So,  $0\leq\tau_n\leq1, 0\leq\theta_n\leq1$ and $\tau_n+\theta_n=1.$ 
Fix  a  K\"ahler metric $g_0$ on $\mathcal M_F$ with K\"ahler form $\omega_0.$  
  Set $g_n=\pi^*g+n^{-2}\theta_{n}g_0.$  Since $g_0$ is  positive  definite on $\mathcal M_F,$ we see  that $g_n$ is also  positive  definite on $\mathcal M_F.$  Moreover, since 
   $d(\pi^*\omega+n^{-2}\theta_{n}\omega_0)=\pi^*d\omega+n^{-2}\theta_{n}d\omega_0=0,$
   we conclude    that $g_n$ is a  K\"ahler metric on $\mathcal M_F$ with   $g_n=\pi^*g$ on $\mathcal M_F\setminus U_n(K).$
   Let   $\tilde m, \tilde m_{n}$ stand for     the Riemannian  measures  
 of $\pi^*g, g_n,$  respectively. 
  Since   ${\rm codim}_{\mathbb C}\mathscr B_F=1$ and  $g_n\to \pi^*g$    as $n\to+\infty,$  
          we   deduce that       $\tilde m_{n}\to \tilde m$ weakly as $n\to+\infty.$
                      Let $\tilde\Delta, \tilde{\Delta}_n$ 
                         stand for       the Laplace-Beltrami operators
         of $\pi^*g, g_n,$  respectively,  in which   $\tilde{\Delta}$ is well-defined on 
                  $\mathcal M_F\setminus K.$ 
                                   According to  $g_n\to \pi^*g$  and   $\tilde m_{n}\to \tilde m$ weakly as $n\to+\infty,$ 
         we   obtain    $\tilde\Delta_{n}\to\tilde\Delta$ weakly as $n\to+\infty.$        
Taking a   precompact  domain $\tilde\Omega$ with smooth boundary $\partial\tilde\Omega$ in  a leaf of $\mathcal M_F.$
We fix a reference point   $\tilde o\in\tilde\Omega$ such that  $\tilde o\not\in K.$
 Denote by    $\tilde g_{\tilde\Omega}(\tilde o, x), \tilde g_{\tilde\Omega, n}(\tilde o, x)$    
   the  Green functions  of  $\tilde\Delta/2, \tilde{\Delta}_n/2$ for
             $\tilde\Omega$ with a pole at $\tilde o$ satisfying Dirichlet boundary condition, respectively,  in which    
                                              $\tilde g_{\tilde\Omega}(\tilde o, x)$ is well-defined  on $\overline{\tilde\Omega}\setminus K.$
                                          Again,   we denote by      $\tilde\pi_{\partial\tilde\Omega}, \tilde\pi_{\partial \tilde\Omega, n}$       
                                      the harmonic measures of $\pi^*g, g_n$ on 
                                                                           $\partial\tilde\Omega$ with respect to $\tilde o$  respectively,  
                                in which      $\tilde\pi_{\partial\tilde\Omega}$  is well-defined  on 
                                                                    $\partial{\tilde\Omega}\setminus K.$ Since  $\mathscr B_F$ is a  polar set, 
                                                                          the previous   arguments can  follow   that 
 $\tilde\Delta_{n}\tilde g_{\tilde\Omega, n}(\tilde o, x)\to \tilde\Delta\tilde g_{\tilde\Omega}(\tilde o, x)$   in the sense of 
  distributions and   
 $\tilde\pi_{\partial \tilde\Omega, n}\to \tilde\pi_{\partial \tilde\Omega}$ weakly  as $n\to+\infty.$  
Applying    Green-Dynkin formula, we get 
$$\int_{\partial \tilde\Omega}u(x)d\tilde\pi_{\partial \tilde\Omega, n}(x)-u(\tilde o)=\frac{1}{2}\int_{\tilde\Omega}\tilde g_{\tilde\Omega, n}(\tilde o,x)\tilde{\Delta}_n u(x)d\tilde v_n(x), \ \ \ \  ^\forall n\geq1$$
 for a $\mathscr C^2$-class function $u$ with that  $u(\tilde o)\not=\infty$ on  $\mathcal M_F$ outside a polar set of singularities at most.  
 Since $\mathscr B_F$ is polar,   we  derive  by letting $n\to+\infty$ that   
 $$\int_{\partial \tilde\Omega}u(x)d\tilde\pi_{\partial \tilde\Omega}(x)-u(\tilde o)=\frac{1}{2}\int_{\tilde\Omega}\tilde g_{\tilde\Omega}(\tilde o,x)\tilde{\Delta} u(x)d\tilde v(x).$$
 It implies    that Dynkin formula works on each leaf of $\mathcal M_F$  under the pullback metric $\pi^*g.$
    According to    Corollary \ref{uniform},   $F$ can  lift to a meromorphic mapping $f: \mathcal M_F\to N$ 
   such that $F=f\circ\pi^{-1}.$  In further, it is not difficult  to examine  that 
 Theorem \ref{first},  Theorem \ref{main1} and Theorem \ref{main2}
 also apply  to $f$  on each leaf of $\mathcal M_F$ under the corresponding conditions. 

\noindent\textbf{{A. Non-positively Curved Case}}~

 With the pullback metric $\pi^*g,$ 
we  see  that  $\mathcal M_F$ has  non-positive  sectional  curvature outside $\pi^{-1}(\mathscr B_F).$
     Let $\tilde\alpha=\pi^*\alpha$ and $\tilde{\mathscr R}=\pi^*{\mathscr R}$  be the K\"ahler form and Ricci form of $\pi^*g,$ respectively,  which are well-defined on $\mathcal M_F\setminus \pi^{-1}(\mathscr B_F).$  
        Given an effective  divisor  $D$  cohomologous to $\omega$ with   associated  function $u_D.$ 
       Review  the Nenvanlinna's functions of $F$ by setting $\Delta(r)=B(r)$  in Section \ref{61}.  

  Let  $\tilde B_1(r),\cdots,\tilde B_\nu(r)$ stand for   the    
  $\nu$ connected components of $\pi^{-1}(B(r))$   centered   at $\tilde o_1,\cdots,\tilde o_\nu$  
with that $\pi(\tilde o_k)=o$ for each $k,$ respectively. 
Define 
$$\tilde g_r(\tilde o_k, x)=g_r(o,\pi(x)), \ \ \ \  d\tilde \pi_{k, r}(x)=d\pi_r(\pi(x)).$$
Note that         outside  $\pi^{-1}(\mathscr B_F),$      $\tilde g_r(\tilde o_k,x)$ is  the  Green function of $\tilde\Delta/2$ for $\tilde B_k(r)$ with a pole at $\tilde o_k$ satisfying  Dirichlet boundary condition,  
and moreover,    $\tilde\pi_{k, r}$ is  the harmonic measure on  $\partial{\tilde B}_k(r)$ with respect to $\tilde o_k.$ 
For each $k,$ put  
  \begin{eqnarray*}
  T_{f, k}(r, \omega)&=& \frac{\pi^m}{(m-1)!}\int_{\tilde B_k(r)}\tilde g_{r}(\tilde o_k,x) f^*\omega\wedge\tilde\alpha^{m-1}, \\
  m_{f,k}(r, D)&=&\int_{\partial{\tilde B}_k(r)} f^*u_Dd\tilde\pi_{r}, \\
        N_{f, k}(r, D)&=&\frac{\pi^m}{(m-1)!}\int_{f^*(D)\cap \tilde B_k(r)}\tilde g_{r}(\tilde o_k,x) \tilde \alpha^{m-1},  \\
      \overline{N}_{f, k}(r, D)&=&\frac{\pi^m}{(m-1)!}\int_{f^{-1}(D)\cap \tilde B_k(r)}\tilde g_{r}(\tilde o_k,x) \tilde \alpha^{m-1}
    \end{eqnarray*}
and
$$T_k(r, \tilde{\mathscr R})= \frac{\pi^m}{(m-1)!}\int_{\tilde B_k(r)}\tilde g_{r}(\tilde o_k,x) \tilde{\mathscr R}\wedge\tilde\alpha^{m-1}.$$
   When $F$ is differentiably non-degenerate,  i.e., at least one single-valued component of $F$ is differentiably non-degenerate, 
we can yield     from  Theorem \ref{main1} that for effective divisors $D_1,\cdots,D_q$ in general position on $N$ such that each $D_j$ is cohomologous to  $\omega,$ 
there is  a subset $E_\delta\subset(0, \infty)$ of  finite Lebesgue measure such that for $k=1,\cdots,\nu$
        \begin{eqnarray}\label{bcg}
&&qT_{f, k}(r,\omega)+T_{f,k}(r, K_N)+T_k(r,  \tilde{\mathscr R}) \\
&\leq& \sum_{j=1}^q\overline N_{f,k}(r,D_j)+O\big(\log T_{f,k}(r, \omega)+\sqrt{-\kappa_k(r)}r+\delta\log r\big) \nonumber
         \end{eqnarray}
    holds for all  $r>1$ outside $E_\delta,$  where $\kappa_k$ is similarly defined as in $(\ref{ricci})$  outside $\pi^{-1}(\mathscr B_F).$          

  \begin{theorem}[Second Main Theorem]\label{main3}  Let $M$ be a  
     complete non-compact  K\"ahler manifold   with non-positive sectional curvature.  
Let $(N,\omega)$ be a   K\"ahler  projective manifold  of  dimension not greater than  that  of $M.$ 
  Let $D_1,\cdots,D_q$ be effective divisors in general position on $N$ such that each $D_j$ is cohomologous to  $\omega.$
  Let  $F: M\rightarrow N$ be a  differentiably non-degenerate $\nu$-valued algebroid mapping.  
  Assume that $q\omega-{\rm Ric}(\omega^n)>0.$
  Then  for any  $\delta>0,$ there exists a subset $E_\delta\subset(1, \infty)$ of finite Lebesgue measure such that 
        \begin{eqnarray*}
&&\left(q-(2\nu-2)\left[\frac{\omega_{FS}|_N}{\omega}\right]\right)T_F(r,\omega)+T_F(r, K_N)+T(r, \mathscr R) \\
&\leq& \sum_{j=1}^q\overline N_F(r,D_j)+O\big(\log T_{F}(r, \omega)+\sqrt{-\kappa(r)}r+\delta\log r\big)
         \end{eqnarray*}
    holds for all  $r>1$ outside $E_\delta,$  where $\kappa$ is defined by $(\ref{ricci}).$          
\end{theorem}  

\begin{proof} It is noticed    that $\pi: \mathcal M_F\setminus{\pi^{-1}(\mathscr B_F)}\rightarrow M\setminus{\mathscr B_F}$  is a $\nu$-sheeted  analytic covering,  
and   two distinct connected components 
  $\tilde B_i(r), \tilde B_j(r)$  of $\pi^{-1}(B(r))$  meet   only at     points  in   
  $\pi^{-1}(\mathscr B_F).$  In addition, 
                                              any  point  in  $F^{-1}(D)$ is possibly  a branch point of $w,$
                                               or   any  point  in  $f^{-1}(D)$ is possibly  a ramification  point of $\pi,$ where $D=D_1+\cdots+D_q.$ 
                                     Hence, it yields  from    $F=f\circ\pi^{-1}$ that   
$$\frac{1}{\nu}\sum_{k=1}^\nu \sum_{j=1}^q\overline{N}_{f,k}(r, D_j)\leq \sum_{j=1}^q\overline{N}_F(r, D_j)+ N_{\rm{bran}}(r, F).$$
Note that  $T(r, \mathscr R)=T_{k}(r,  \tilde{\mathscr R})$ as well as   
$$T_F(r, \omega) \leq  \frac{1}{\nu}\sum_{k=1}^\nu T_{f,k}(r,\omega)\leq T_F(r, \omega)+N_{\rm{bran}}(r,F).$$
Moreover, since  the curvature function on  $M$ is continuous  and  the   curvature of $\pi^*g$ outside $\pi^{-1}(\mathscr B_F)$   matches  the curvature of $g$  
outside $\mathscr B_F$ locally, 
one obtains    $\kappa_k=\kappa$ for  $k=1,\cdots,\nu.$
Put together  the above with (\ref{bcg}),    then 
  for any $\delta>0,$ there exists  a subset $E_\delta\subset(0, \infty)$ of finite Lebesgue measure such that 
   \begin{eqnarray*}
&& qT_F(r, \omega)+T_F(r, K_N)+T(r, \mathscr R)  \\
&\leq& \frac{1}{\nu}\sum_{k=1}^\nu \left(qT_{f,k}(r,\omega)+ T_{f,k}(r, K_N)+ T_{k}(r,  \tilde{\mathscr R})\right) \\
&\leq& \frac{1}{\nu}\sum_{k=1}^\nu\bigg(\sum_{j=1}^q\overline{N}_{f, k}(r, D_j)
+O\big(\log T_{f, k}(r,\omega)+\sqrt{-\kappa_k(r)}r+\delta\log r\big)\bigg) \\
&\leq& \sum_{j=1}^q\overline{N}_{F}(r, D_j)+ N_{\rm{bran}}(r, F) \\
&&  +~O\Big(\log\big(T_F(r, \omega)+N_{\rm{bran}}(r,F)\big)+\sqrt{-\kappa(r)}r+\delta\log r\Big) 
   \end{eqnarray*}
holds for all $r>0$ outside  $E_\delta.$ 
Using  Theorem \ref{bran}, then we   have the theorem proved. 
\end{proof}

Let $\mathscr O(1)$ stand for  the hyperplane line bundle over $\mathbb P^n(\mathbb C).$  
It yields from $c_1(K_{\mathbb P^n(\mathbb C)})=-(n+1)c_1(\mathscr O(1))$ that 
  \begin{cor}\label{hco1}    Let $M$ be a  
     complete non-compact  K\"ahler manifold   with non-positive sectional curvature. 
  Let $H_1,\cdots, H_q$ be hyperplanes  of  $\mathbb P^n(\mathbb C)$ in general position.  
  Let  $F: M\rightarrow \mathbb P^n(\mathbb C)$    be a  differentiably non-degenerate $\nu$-valued algebroid mapping.  
 Assume that $m\geq n.$ Then  for any  $\delta>0,$ there exists a subset $E_\delta\subset(1, \infty)$ of finite Lebesgue measure such that 
        \begin{eqnarray*}
&& \left(q-2\nu-n+1\right)T_F(r,\omega_{FS}) +T(r, \mathscr R) \\
&\leq& \sum_{j=1}^q\overline N_F(r,H_j)+O\big(\log T_{F}(r, \omega_{FS})+\sqrt{-\kappa(r)}r+\delta\log r\big)
        \end{eqnarray*}
    holds for all  $r>1$ outside $E_\delta,$  where $\kappa$ is defined by $(\ref{ricci}).$          
\end{cor}

\noindent\textbf{{B. Ricci Non-negatively Curved  Case}}~

Assume that     $M$ is non-parabolic.  
Under the  pullback metric $\pi^*g,$  one can see that $\mathcal M_F$ is also  non-parabolic,  with non-negative Ricci curvature outside $\pi^{-1}(\mathscr B_F).$
 Similarly, 
 we can have      $T_F(r, \omega), m_F(r, D), N_F(r, D), \overline{N}_F(r,D)$ and $T(r, \mathscr R)$ on $\Delta(r)$ defined by (\ref{delta}) in Section \ref{sec421}. 
   By     the similar  arguments as in the proof of Theorem \ref{main3}, we also conclude  from  Theorem \ref{main2} that 

  \begin{theorem}[Second Main Theorem]\label{main4}  Let $M$ be a  non-parabolic 
     complete non-compact  K\"ahler manifold   with non-negative  Ricci curvature.  
Let $(N,\omega)$ be a K\"ahler  projective manifold  of  dimension not greater than  that  of $M.$ 
 Let $D_1,\cdots,D_q$ be effective divisors in general position on $N$ such that each $D_j$ is cohomologous to  $\omega.$
  Let  $F: M\rightarrow N$ be a  differentiably non-degenerate $\nu$-valued algebroid  mapping. 
    Assume that $q\omega-{\rm Ric}(\omega^n)>0.$
   Then  for any  $\delta>0,$ there exists a subset $E_\delta\subset(1, \infty)$ of finite Lebesgue measure such that 
        \begin{eqnarray*}
&&\left(q-(2\nu-2)\left[\frac{\omega_{FS}|_N}{\omega}\right]\right)T_F(r,\omega)+T_F(r, K_N)+T(r, \mathscr R) \\
&\leq& \sum_{j=1}^q\overline N_F(r,D_j)+O\big(\log T_{F}(r, \omega)+\log H(r)+\delta\log r\big)
         \end{eqnarray*}
    holds for all  $r>1$ outside $E_\delta,$  where $H$ is defined by $(\ref{Hr}).$          
\end{theorem}  

Similarly, we have: 
  \begin{cor}\label{hco2}  Let $M$ be a  non-parabolic 
     complete non-compact  K\"ahler manifold   with non-negative  Ricci curvature.  
  Let $H_1,\cdots, H_q$ be hyperplanes   of $\mathbb P^n(\mathbb C)$ in general position.  
  Let  $F: M\rightarrow \mathbb P^n(\mathbb C)$ be a  differentiably non-degenerate $\nu$-valued algebroid  mapping. 
  Assume that $m\geq n.$ Then  for any  $\delta>0,$ there exists a subset $E_\delta\subset(1, \infty)$ of finite Lebesgue measure such that 
     \begin{eqnarray*}
&& \left(q-2\nu-n+1\right)T_F(r,\omega_{FS}) +T(r, \mathscr R) \\
&\leq& \sum_{j=1}^q\overline N_F(r,H_j)+O\big(\log T_{F}(r, \omega_{FS})+\log H(r)+\delta\log r\big)
        \end{eqnarray*}
    holds for all  $r>1$ outside $E_\delta,$  where $H$ is defined by $(\ref{Hr}).$          
\end{cor}


\vskip\baselineskip


\begin{thebibliography}{99}
\vskip\baselineskip
\bibitem{at2018a} A. Atsuji,  Nevanlinna-type theorems for meromorphic functions on non-positively curved K\"ahler manifolds, Forum Math. (1) \textbf{30} (2018), 171-189.
\bibitem{gri} J. Carlson and P. Griffiths, A defect relation for equidimensional holomorphic mappings between algebraic varieties, Ann. Math. \textbf{95} (1972), 557-584. 
\bibitem{C-Y} S. Y. Cheng and S. T. Yau, Differential equations on Riemannian manifolds and their geometric applications, Comm. Pure Appl. Math. \textbf{28} (1975), 333-354.
\bibitem{Deb} A. Debiard, B. Gaveau and E. Mazet, Theorems de comparaison en geometrie Riemannienne, Publ. Res. Inst. Math. Sci. Kyoto, \textbf{12} (1976),
390-425;  \textbf{111} (2019), 303-314.
\bibitem{Dong}  X. J. Dong, Carlson-Griffiths theory for complete K\"ahler manifolds, J. Inst. Math. Jussieu, \textbf{22} (2023), 2337-2365.
\bibitem{DY} X. J. Dong and  S. S. Yang, Nevanlinna theory via holomorphic forms, Pacific J. Math.  (1) \textbf{319} (2022), 55-74. 
\bibitem{HY} P. C. Hu and C. C. Yang, The second main theorem for algebroid functions, Math. Z. \textbf{220} (1995), 99-126.
\bibitem{H-X} Y. Z. He and X. Z. Xiao, Algebroid functions and ordinary differential equations (Chinese),  2nd, Science Press, Beijing,  China, (2016). 
\bibitem{Laine} I. Laine, Nevanlinna Theory and Complex Differential Equations,  De Gruyter Studies in Math. (1993).
\bibitem{Li-Yau} P. Li and S. T. Yau,  On the parabolic kernel of the Schr\"odinger operator, Acta Math.  \textbf{156} (1986), 153-201.
\bibitem{No0} J. Noguchi, On the deficiencies and the existence of Picard's exceptional values of entire algebroid functions, Kodai. Math. Sem. Report, \textbf{26} (1974), 29-35.
\bibitem{No} J. Noguchi  and J. Winkelmann, Nevanlinna theory in several complex variables and Diophantine approximation,  A series of comprehensive studies in mathematics, Springer, (2014).
\bibitem{NO} K. Niino and M. Ozawa,  Deficiencies of an entire algebroid function, Kodai Math. Sem. Report, \textbf{22} (1970), 98-113.
 \bibitem{Oz} M. Ozawa, On the growth of algebroid functions with several deficiencies II, Kodai Math. Sem. Rep., \textbf{22} (1970), 129-137. 
 \bibitem{P1} A. C. Peterson, Comparison theorems and existence theorem for ordinary differential equations, J. Math. Anal. Appl. \textbf{55} (1976), 773-784. 
  \bibitem{GR} G. R\'emoundos, Extension aux fonctions alg\'ebroides multipliformes du th\'eor\`eme de M. Picard et de ses g\'en\'eralisations, 
M\'em. Sci. Math., fase Paris, Gauthier-Villar, \textbf{23} (1927).
 \bibitem{ru} M. Ru, Nevanlinna Theory and Its Relation to Diophantine Approximation, 2nd edn. World Scientific Publishing, (2021). 
 \bibitem{ru0} M. Ru, Algebroid functions, Wirsing's theorem and their relations, Math. Z. \textbf{233} (2000), 137-148.
 \bibitem{Sha} B. V. Shabat, Distribution of values of holomorphic mappings, Amer. Math. Soc.  (1985).
 \bibitem{S-G} D. C. Sun and Z. S. Gao, Value distribution of algebroid functions (Chinese),  Science Press, Beijing,  China, (2014). 
 \bibitem{JS} J. Suzuki, On picard values of algebroid functions in a neighborhood of a totally disconnected compact set,   Nagoya Math. J. \textbf{40} (1970), 1-12. 
 \bibitem{LS1} L. Selberg, \"Uber die Wertverteilung der Algebroide Funktionen, Math. Z. \textbf{31} (1930), 709-728. 
\bibitem{LS2} L. Selberg, Algebroide Funktionen und Umkehrfunktionen Abelscher Integrale, Ark. Norskes Vid. Akad. Oslo, \textbf{8} (1943), 1-72. 
\bibitem{S-Y} R. Schoen and S. T. Yau,  Lectures on Differential Geometry,  International Press, (2010). 
\bibitem{Sa0} T. Sasaki, On the Green Function of a Complete Riemannian or K\"ahler manifold with Asymptotically Negative Constant Curvature and Applications, Adv. Stud. in Pure Math. \textbf{3} (1984), 387-421.
\bibitem{To1} N. Toda, Sur l'ensemble d'adh\'erence fine des fonctions alg\'ebroldes, Nagoya Math. J. \textbf{30} (1966), 295-302. 
\bibitem{To2} N. Toda, Sur  les directions de Julia et Borel des fonctions alg\'ebroldes,  Nagoya Math. J. \textbf{34} (1969), 1-23. 
\bibitem{EU} E. Ullrich, \"Uber den Einfluss der Verzweightheit einer Algebroide auf ihre Wertverteilung. J. rei. u. ang. Math. \textbf{167} (1931), 198-220. 
\bibitem{GV}  G. Valiron, Sur quelques  propri\'et\'es des fonctions alg\'ebroides, C. R. Acad. Sci. Paris, \textbf{189} (1927), 824-826. 
\end{thebibliography}
\end{document}